\documentclass[10pt]{amsart}
\usepackage{amssymb}
\usepackage{amsmath}
\usepackage{latexsym,enumerate}
\usepackage{enumerate,graphicx}
\usepackage{array,subfig}
\usepackage{xcolor}
\usepackage[ulem=normalem]{changes}
\usepackage{dsfont}
\usepackage{psfrag}
\usepackage{enumerate}

\definechangesauthor{CB}{red}
\definechangesauthor{EC}{blue}
\definechangesauthor{FQ}{green}
\setauthormarkup[]{}

\newtheorem{theorem}{Theorem}[section]
\newtheorem{lemma}[theorem]{Lemma}
\newtheorem{proposition}[theorem]{Proposition}
\newtheorem{corollary}[theorem]{Corollary}
\theoremstyle{definition}
\newtheorem{definition}[theorem]{Definition}
\theoremstyle{remark}

\newcommand{\supp}{\operatorname{supp}}

\newcommand{\dt}{\,\mathrm{d}t}
\newcommand{\dx}{\,\mathrm{d}x}
\newcommand{\dy}{\,\mathrm{d}y}
\newcommand{\dz}{\,\mathrm{d}z}
\newcommand{\ds}{\,\mathrm{d}s}
\newcommand{\dtau}{\,\mathrm{d}\tau}
\newcommand{\eps}{\ensuremath{\varepsilon}}
\newcommand{\R}{\ensuremath{\mathbb{R}}}
\newcommand{\N}{\ensuremath{\mathbb{N}}}

\newcommand{\sign}{\mathop{\rm sign}}
\newcommand{\dist}{\mathop{\rm dist}}
\newcommand{\e}{\mathrm{e}\,}
\renewcommand{\L}{\mathrm{L}}
\newcommand{\C}{\mathrm{C}}

\newcommand{\remark}{\noindent \emph{Remark. }}

\renewcommand{\leq}{\leqslant}
\renewcommand{\geq}{\geqslant}
\renewcommand{\le}{\leqslant}
\renewcommand{\ge}{\geqslant}

\newcommand{\vvvert}{|\hspace*{-1pt}|\hspace*{-1pt}|}

\numberwithin{equation}{section}
\newlength{\figwidth}
\title[]{A nonlocal one-phase Stefan problem that develops mushy regions}

\author[Br\"andle]{Cristina Br\"andle}
\address[Cristina Br\"{a}ndle]{Departamento de Matem\'{a}ticas, Universidad Carlos III de Madrid,
28911 Legan\'{e}s, Spain}
\email{cbrandle\@@math.uc3m.es}

\author[Chasseigne]{Emmanuel Chasseigne}
\address[Emmanuel Chasseigne]{Laboratoire de Math\'ematiques et Physique Th\'eorique,
CNRS UMR 6083, F\'ed\'eration Denis Poisson,
Universit\'e Fran\c{c}ois Rabelais, Parc de Grandmont,
37200 Tours, France}
\email{manu\@@lmpt.univ-tours.fr}

\author[Quir\'os]{Fernando Quir\'os}
\address[Fernando Quir\'{o}s]{Departamento de Matem\'{a}ticas, Universidad Aut\'{o}noma de Madrid
28049 Madrid, Spain}
\email{fernando.quiros\@@uam.es}

\thanks{All the authors supported by
Spanish Projects MTM2008-06326-C02-01 and -02}



\begin{document}
\begin{abstract}
We study a nonlocal version of the one-phase Stefan problem which
develops mushy regions, even if they were not present initially, a
model which can be of interest at the mesoscopic scale. The equation
involves a convolution with a compactly supported kernel. The
created mushy regions have the size of the support of this kernel.
If the kernel is suitably rescaled, such regions disappear and the
solution converges to the solution of the usual local version of the one-phase
Stefan problem. We prove that the model is well posed, and give
several qualitative properties. In particular, the long-time
behavior is identified by means of a nonlocal mesa solving an
obstacle problem.
\end{abstract}

\maketitle

\section{Introduction}
The aim of this paper is to study the following nonlocal version of
the one-phase Stefan problem posed in $\mathbb{R}^N\times(0,\infty)$,
\begin{equation}
\label{eq:stefan.nolocal} \left\{
\begin{aligned}
  &\partial_tu=J\ast v-v,\text{ where }v=(u-1)_+,\\
  &u(\cdot,0)=f\ge0,
\end{aligned}
\right.
\end{equation}
which presents interesting features from the physical point of view.
The function $J$ is assumed to be continuous, compactly supported,
radially symmetric and with $\int_{\mathbb{R}^N}J=1$. We denote by
$R_J$ the radius of the support of $J$. For some results we will
also assume that $J$ is nonincreasing in the radial variable.

\noindent\textsc{The local model --} The well-known usual local Stefan problem is a mathematical model
that describes the phenomenon of phase transition, for example
between water and ice, \cite{Me}, \cite{Ru}.
Its history goes back to Lam\'{e} and
Clapeyron \cite{LC} and, afterwards, Stefan \cite{St}. The one-phase Stefan
problem corresponds to the simplified case in which the temperature of the ice phase is supposed to
be maintained at the value where the phase transition occurs, say
$0^{\circ}{\rm C}$. The thermodynamical
state of the system is characterized by two state variables, {\it
temperature} $v$ and {\it enthalpy} $u$. Conservation of energy
implies  that they satisfy, in the absence of heat sources or sinks,
the evolution equation
$$
\rho \partial_t u= \nabla \cdot(\kappa \nabla v),
$$
where the {\it density} $\rho>0$ and the {\it thermal
conductivity} $\kappa>0$ are assumed to be constant.

On the other hand, there is a constitutive equation relating $v$ to
$u$, given in an ideal situation by the formula
\begin{equation}\label{eq:relation.v.u}
v=c^{-1}(u-L)_+,
\end{equation}
where $c>0$, which we assume to be constant, is called the {\it
specific heat} (the amount of energy needed to increase in one unit
the temperature of a mass unit of water) and $L>0$ is the {\it
latent heat} (the amount of energy needed to transform a mass unit
of ice into water). All the parameters  can be set to one with a
change of units, and we arrive to
\begin{equation}
\label{eq:stefan.local}
\partial_t u=\Delta v,\qquad
v=(u-1)_+.
\end{equation}
In contrast with the standard heat equation, this problem
has \emph{finite speed of propagation}: if the initial data are
compactly supported, the same is true for $u$ (and hence for $v$)
for any later time. This is one of the main features of the model,
and gives rise to the existence of two \emph{free boundaries}, one
for $u$, the set $\partial\{u>0\}$, and one for $v$,
$\partial\{v>0\}$.

On the ice we have $u=0$, while $u\ge 1$ on the liquid phase. The
points where $0<u<1$ correspond to the \emph{mushy region}, where we have neither ice nor water, but
something in an intermediate state between solid and liquid. The
temperature in this zone is $0$, but not its enthalpy.

There is a major drawback in the model: either $u(x,t)\ge1$ or
$u(x,t)=f(x)$~\cite{Friedman1968}. This means, on one hand, that a
point can belong to the  mushy region only if it belonged to it at
the initial time. On the other hand, there is no evolution of the
enthalpy inside the mushy region. Both things are unsatisfactory
from the physical point of view. Indeed, once the structure of ice
starts to break, it should take some time until it melts completely.
The region where this occurs may be small, but it should be noticed
at some intermediate (mesoscopic) scale between the microscopic and
macroscopic ones.

\noindent\textsc{The nonlocal model --} Bearing the above discussion
in mind, we consider the nonlocal diffusion version
\eqref{eq:stefan.nolocal} of the Stefan problem. As we shall see,
this model also has
finite speed of propagation, and hence free boundaries. Moreover, in
sharp contrast with the local problem, it develops mushy regions,
even if they were not present initially. The size of these mushy
regions is given by the size of the support of the convolution
kernel $J$. Finally, if the initial data are continuous, the
solution remains continuous for all times.

It is possible to rescale the kernel in such a way that solutions of
the nonlocal model converge to solutions of the local
version~\eqref{eq:stefan.local}. In this scaling, the support of the
kernel shrinks to a point, and hence the mushy regions disappear.
Thus, the local model can be viewed as a limit problem when going
from the intermediate scale to the macroscopic one.

What is the point in considering this particular  diffusion
operator, $\mathcal{L}v:=J\ast v -v$? With this choice, the evolution
of $u$ at a certain point is governed by a balance between the value of the temperature
at this point and a certain average of this physical magnitude in a fixed neighborhood. This
is a way to take into account possible middle-range interactions
between water and ice. In the local model, the region where the
average is taken shrinks to a point, and  the evolution of the
enthalpy is governed by the Laplacian of the temperature.

\noindent\textsc{Choice of the kernel --} One could use other non-local operators,
as fractional-Laplacian type operators of the form $\mathcal{K}v:=-(-\Delta)^{s}v$.
However, we will stick to operators $\mathcal{L}$ involving kernels which are both compactly supported and non-singular, though
most (if not all) of our results should be valid for singular
(yet compactly supported) kernels.
 Why are we being so restrictive? There are two main reasons:
\begin{itemize}
\item[$(i)$] Using a  non-compactly supported kernel means
that \lq\lq infinite-range" interactions are not considered to be negligible.
This would imply for instance that the mushy regions which are automatically created are instantaneously
spread throughout the space (see Section~\ref{subsec.mushy} with
$R_J=\infty$), something which is not satisfactory.\medskip

\item[$(ii)$] Kernels which have a singularity at the origin have
the big advantage of implying regularization properties
(\cite{BarlesChasseigneImbert} and~\cite{CaffarelliSilvestre}).
It is challenging from
a mathematical point of view to try to
handle situations where no regularizing effect can occur.

 Another way to think about this
absence of regularization is to understand it as a lack of
compactness of the inverse of $\mathcal{L}$, which is the source of
several problems. Our approach, which consists in controlling the
supports of solutions, allows us to bypass such problems, up to a
certain point.
\end{itemize}

\noindent\textsc{About the initial data  --}  We would like to
consider typical situations in which initially there is ice with zero
enthalpy everywhere except at some places where there is water at
some positive temperature. In this case, we shall see that
a mushy region will develop as ice
melts from an initial configuration in which no mushy region was
present. Thus, we have to consider initial data which are in
principle not continuous, but only integrable.
This will lead to
solutions that are not continuous. Therefore, we have to handle the
various qualitative properties concerning the supports and mushy
regions not in the usual (continuous) sense, but in the sense of
distributions, which requires some effort.

If the initial data are continuous and bounded, we can construct a
solution with these two properties. Some of our techniques can be
greatly simplified for such solutions, and can be later applied to
more general data by approximation. Hence we will devote some time
to discuss the basic theory for solutions in this class.

As expected, when the initial data are at the same time bounded,
continuous and integrable,  the two concepts of solution mentioned
above coincide.

In the one-phase Stefan problem, the temperature, and hence the enthalpy, are assumed to be nonnegative. Therefore, in principle we should only consider nonnegative initial data $f$.  Indeed, if the initial enthalpy had sign changes,  we would  be dealing with a two-phase Stefan problem, and the relation between the temperature and  the enthalpy would not be given by \eqref{eq:relation.v.u}. In particular, the temperature should be negative for negative values of the enthalpy. Therefore, the relevant model would be different.

Though the problem might not have any physical meaning for functions $f$ that have sign changes, some of our results are still true for such general initial data. Whenever we know that this is the case, we will state the corresponding theorem for a general $f$, and will add the restriction of nonnegativity only if required.

\noindent\textsc{Long-time behavior --} Another important aspect of
evolution equations is the long time behavior of solutions.  For a
certain class of initial data (which does not include all non-negative functions in $\L^1(\mathbb{R}^N)$  if $N=1,2$),
we prove that the solution converges as time
goes to infinity to a solution of a nonlocal mesa problem which is
not the ``classical'' (local) mesa problem.
The identification of the limit (the mesa) can be done by
solving an obstacle problem. More precisely, the asymptotic
behavior of the solution $u$ is given by the projection
$$
{\mathcal P}f:=f+J\ast w-w
$$
where $w$ solves an obstacle problem detailed in
Section~\ref{sect:asymptotic}.

The projection operator ${\mathcal P}$ is $\L^1$-contractive in the set $\mathcal{S}$ of
admissible functions for which the above convergence result is valid, Corollary~\ref{corollary:contraction.mesa}.
On the other hand, $\mathcal{S}$ is dense in the set of $\L^1(\mathbb{R}^N)$ functions which are non-negative a.e.
Therefore,  ${\mathcal P}$  can be extended by continuity to the latter class of functions. This will
allow us to identify the large time limit of solutions of
\eqref{eq:stefan.nolocal} for the wider class of integrable and
nonnegative initial data.

\noindent\textsc{Abstract setting  --} Equations like
\eqref{eq:stefan.nolocal} and \eqref{eq:stefan.local} can be
embedded in the abstract setting of semi-group theory for equations
of the form
$$
    \gamma(u)_t+Au\ni0,
$$
where $\gamma$ is a maximal monotone graph, and $A$ is a linear
(bounded or unbounded) $m$-accretive operator,  see
\cite{CrandallPierre1982} and also
\cite{AndreuMazonRossiToledo2008},
\cite{AndreuMazonRossiToledoBook2010}. This theory provides
existence and uniqueness for our model, though not the main
qualitative properties considered here. As for existence and
uniqueness, in the special case that we are studying, $\gamma$ is
the inverse of the Lipschitz graph $\Gamma:s\mapsto(s-1)_+$,  and
many arguments can be written at a lower cost.

\noindent{\sc Notation -- }
Given a set $\Omega$ we define:
 \begin{itemize}\itemsep=6pt
     \item $\mathrm{BC}(\Omega)=\{\varphi\in \C(\Omega): \varphi\text{ bounded in
        }\mathbb{R}^N\}$;
      \item $\C_\mathrm{c}(\Omega)=\{\varphi\in
          \C(\Omega): \varphi \textrm{ compactly supported}\}$;
    \item $\C_\mathrm{c}^\infty(\Omega)=\{\varphi\in
          \C^\infty(\Omega): \varphi \textrm{ compactly supported}\}$.
   \end{itemize}
We also denote
\begin{itemize}\itemsep=6pt
  \item  $\L^1_+(\mathbb{R}^N)=\{\varphi\in\L^1(\mathbb{R}^N):\varphi\ge0 \text{ a.e.}\}$;
  \item $\mathrm{BC}_+(\mathbb{R}^N)=\{\varphi\in \mathrm{BC}(\mathbb{R}^N): \varphi\ge0\}$;
    \item $\C_0(\mathbb{R}^N)=\{\varphi\in \C(\mathbb{R}^N): \varphi\to 0 \text{ as
  }|x|\to\infty\}$.
\end{itemize}

The elements $\psi$ belonging to $\L^1((0,T);\L^1(\R^N))$ will sometimes be viewed
  as elements of $\L^1(\R^N\times(0,T))$. In such cases we denote $\psi(x,t)=\psi(t)(x)$.

\noindent{\sc Organization of the paper -- } In Section 2 we derive
the general theory of the model both for integrable initial data and
for continuous and bounded initial data. Section 3 is devoted to the
study of mushy regions and free boundaries. Convergence to the local
Stefan problem and disappearance of the mushy regions in the
macroscopic scale are done in Section 4. In Section 5 we study the
large time behavior of solutions. Numerical experiments and
illustrations of the qualitative properties of the model are
collected in Section 6. We devote the last section of the paper to
establish some conclusions and make some comments on the model. In
order to prove some of our results, we have needed to improve
slightly the existing results on the asymptotic behavior of
solutions to the non-local heat equation. Such improvements are
proved in an appendix.

\section{Basic theory of the model}
\label{sect:basic}
\setcounter{equation}{0}

We will develop here the basic theory for the two
concepts of solution mentioned in the introduction.

\subsection{$\L^1$ theory}
We start with the theory for integrable initial data. In this case
the solution is regarded as a continuous curve in $\L^1(\R^N)$.
\begin{definition}
    Let $f\in\L^1(\R^N)$.
    An $\L^1$-solution of~\eqref{eq:stefan.nolocal} is a function $u\in \C([0,\infty);\L^1(\R^N))$ such that ~\eqref{eq:stefan.nolocal}
    holds in the sense of distributions, or equivalently, if for every
    $t>0$, $u(t)\in\L^1(\R^N)$ and
    \begin{equation}\label{eq:form.int}
        u(t)=f+\int_0^t \big(J\ast v(s)-v(s)\big)\ds,\qquad v=(u-1)_+\quad\text{ a.e.}
    \end{equation}
\end{definition}

\remark
If $u$ is an $\L^1$-solution, then
$u\in\L^1(\mathbb{R}^N\times[0,T])$ for all $T>0$.
Hence,~\eqref{eq:stefan.nolocal} holds, not only  in the sense of
distributions, but also~a.e., and $u$ is said to be a \emph{strong}
solution. Moreover, since $v=(u-1)_+\in\C([0,\infty);\L^1(\R^N))$,
we also have $u\in\C^1([0,\infty);\L^1(\R^N))$, and the equation
holds~a.e.~in $x$ for all $t\ge0$.

\begin{theorem}\label{thm:stefan.noloc.L1}
    For any $f\in \L^1(\R^N)$, there exists a unique $\L^1$-solution of~\eqref{eq:stefan.nolocal}.
\end{theorem}
\begin{proof}
Let $\mathcal{B}_{t_0}$ be the Banach space consisting of the functions $u\in\C([0,t_0];\L^1(\mathbb{R}^N))$ endowed with the norm
    $$
      \vvvert u\vvvert=\max_{0\leq t\leq t_0}\|u(t)\|_{\L^1(\mathbb{R}^N)}.
    $$
We define the operator
$\mathcal{T}:\mathcal{B}_{t_0}\to\mathcal{B}_{t_0}$ through
$$
  (\mathcal{T}_{f}u)(t)=f+\int_0^t\left( J*(u-1)_+(s)-(u-1)_+(s)\right)\ds.
$$
This operator turns out to be contractive if $t_0$ is small enough. Indeed,
$$
\begin{array}{l}
 \displaystyle\int_{\mathbb{R}^N}
 |\mathcal{T}_{f}\varphi-\mathcal{T}_{f}\psi|(t)
 \\[10pt]
\displaystyle \qquad
 \leq\int_{\mathbb{R}^N}\int_0^t \Big( |J*((\varphi-1)_+-(\psi-1)_+)(s)|
  +
  |(\varphi-1)_+-(\psi-1)_+|(s)\Big)\ds \\[10pt]
\displaystyle  \qquad \leq
\int_{0}^t\Big(\|J\|_{\L^1(\mathbb{R}^N)}+1)\Big)\|((\varphi-1)_+-(\psi-1)_+)(s)\|_{\L^1({\mathbb{R}^N)}}\ds.
\end{array}
$$
Hence
$$\begin{aligned}
  \vvvert\mathcal{T}_{f}\varphi-\mathcal{T}_{f}\psi\vvvert & \leq 2t_0
  \max_{0\le t\le t_0}\|((\varphi-1)_+-(\psi-1)_+)(t)\|_{\L^1(\mathbb{R}^N)}\\
 & \leq 2t_0\vvvert \varphi-\psi\vvvert.
\end{aligned}
$$
Thus, $\mathcal{T}$ is a contraction if $t_0<1/2$.
Existence and uniqueness in the time interval $[0,t_0]$ now follow easily, using Banach's fixed point theorem. Since the length of the  existence and uniqueness time interval does not depend on the initial data, we may iterate the argument to extend the result to all positive times.
\end{proof}

The energy of the $\L^1$-solutions is constant in time.
\begin{theorem}
 \label{corol:mass.conservation}
 Let $f\in \L^1(\R^N)$. The $\L^1$-solution $u$ to~\eqref{eq:stefan.nolocal} satisfies
    $$\int_{\R^N}u(t)=\int_{\R^N}f \qquad\text{for every }t>0.
    $$
\end{theorem}
\begin{proof}
    Since $u(t)\in \L^1(\R^N)$ for any $t$, integration of the equation~\eqref{eq:form.int} in space
    yields, thanks to Fubini's Theorem:
    $$\int_{\R^N} u(t) = \int_{\R^N} f + \int_0^t \Big(\int_{\R^N} J\ast v(s) - \int_{\R^N} v(s)\Big)\ds=\int_{\R^N} f.
    $$
\end{proof}

Our next aim is to derive an $\L^1$-contraction property for $\L^1$-solutions.  In order to obtain it, we need first to approximate the
graph $\Gamma(s)=(s-1)_+$ by a sequence of strictly monotone graphs
$\Gamma_n(s)$ such that:
\begin{itemize}
        \item[(i)] there is a constant $L$ independent of $n$ such that $|\Gamma_n(s)-\Gamma_n(t)|\le L|s-t|$, for all $n\in\mathbb{N}$ and
        $s,t\ge0$;
        \item[(ii)] for all  $
         n\in\mathbb{N}$, $\Gamma_n(0)=0$ and $\Gamma_n$ is strictly increasing on $[0,\infty)$;
        \item[(iii)]  $\Gamma_n(s)\leq s$ for all $n\in\mathbb{N}$ and $s\ge 0$;
        \item[(iv)]  $\Gamma_n\to\Gamma$ as
         $n\to\infty$ uniformly in $[0,\infty)$;
    \end{itemize}
(take for instance $\Gamma_n(s)=s/(n+1)$ for $0\leq s\leq (n+1)/n$,
and $\Gamma_n(s)=s-1$ for $s>(n+1)/n$).

Since $\Gamma_n$ is Lipschitz, for any $f\in\L^1(\R^N)$ and any
$n\in\N$ there exists a unique $\L^1$-solution
$u_n\in\C([0,\infty);\L^1(\mathbb{R}^N))$ of the approximate problem
\begin{equation}
\label{eq:approximate.problems}
  \partial_t u_n = J\ast \Gamma_n(u_n)- \Gamma_n(u_n)
\end{equation}
with initial data $u_n(0)=f$. The proof is just like that of
Theorem~\ref{thm:stefan.noloc.L1}.
Moreover,   $\Gamma_n(u_{n})\in\C([0,\infty);\L^1(\mathbb{R}^N))$, and, hence,
    $u_{n}\in\C^1([0,\infty);\L^1(\mathbb{R}^N))$.
    Conservation of energy also holds.

The $\L^1$-contraction for our original problem will follow from an
analogous result for the approximate problems.
\begin{lemma}
    \label{lem:comp.approx}
  Let $u_{n,1}$ and $u_{n,2}$ be two $\L^1$-solutions of~\eqref{eq:approximate.problems}
  with initial data $f_{1},f_{2}\in \L^1(\R^N)$. Then,
        \begin{equation}
        \label{eq:contraction.approximate}
        \int_{\R^N}\big(u_{n,1}-u_{n,2}\big)_+(t)\leq
        \int_{\R^N}\big(f_{1}-f_{2}\big)_+ \qquad\text{for every }t\ge0.
        \end{equation}
\end{lemma}
\begin{proof} We subtract  the equations for $u_{n,1}$ and $u_{n,2}$
    and multiply by $\mathds{1}_{\{u_{n,1}>u_{n,2}\}}$.
    Since $u_{n,1}-u_{n,2}\in\C^1([0,\infty);\L^1(\mathbb{R}^N))$, then
    $$
    \partial_t (u_{n,1}-u_{n,2})\mathds{1}_{\{u_{n,1}>u_{n,2}\}}=\partial_t (u_{n,1}-u_{n,2})_+.
    $$
    On the other hand, since $0\leq\mathds{1}_{\{u_{n,1}>u_{n,2}\}}\leq 1$, we have
    $$
    J\ast(\Gamma_n(u_{n,1})-\Gamma_n(u_{n,2}))\mathds{1}_{\{u_{n,1}>u_{n,2}\}}
    \leq J\ast(\Gamma_n(u_{n,1})-\Gamma_n(u_{n,2}))_+.
    $$
    Finally, since $\Gamma_n$ is strictly monotone,
    $\mathds{1}_{\{u_{n,1}>u_{n,2}\}}=\mathds{1}_{\{\Gamma_n(u_{n,1})>\Gamma_n(u_{n,2})\}}$.
    Thus,
    $$
    (\Gamma_n(u_{n,1})-\Gamma_n(u_{n,2}))\mathds{1}_{\{u_{n,1}>u_{n,2}\}}=(\Gamma_n(u_{n,1})-\Gamma_n(u_{n,2}))_+
    .
    $$
    We end up with
    $$
    \partial_t (u_{n,1}-u_{n,2})_+\leq J\ast(\Gamma_n(u_{n,1})-\Gamma_n(u_{n,2}))_+-
    (\Gamma_n(u_{n,1})-\Gamma_n(u_{n,2}))_+.
    $$
    Integrating in space, and using Fubini's Theorem, which can be
    applied, since
    $(\Gamma_n(u_{n,1}(t))-\Gamma_n(u_{n,2}(t)))_+\in\L^1(\R^N)$, we get
    $$
    \partial_t \int_{\mathbb{R}^N}(u_{n,1}-u_{n,2})_+(t)\le0.
    $$
\end{proof}

\remark
It follows immediately from \eqref{eq:contraction.approximate} that
$$
    \|(u_{n,1}-u_{n,2})(t)\|_{\L^1(\mathbb{R}^N)}\leq  \|f_{1}-f_{2}\|_{\L^1(\mathbb{R}^N)}.
$$

\begin{corollary}\label{cor:contraction.stefan.noloc}
  Let $u_1$ and $u_2$ be two $\L^1$-solutions of~\eqref{eq:stefan.nolocal}
  with initial data $f_{1},f_{2}\in \L^1(\R^N)$. Then, for every $t\ge0$,
  \begin{equation}
  \label{eq:contraction.enthalpy}
  \int_{\mathbb{R}^N} (u_1-u_2)_+(t)\leq   \int_{\mathbb{R}^N} (f_{1}-f_{2})_+.
  \end{equation}
\end{corollary}\label{cor:L1.comp}
\begin{proof}
The idea is to pass to the limit in $n$ in the contraction
property~\eqref{eq:contraction.approximate} for the approximate
problems~\eqref{eq:approximate.problems}. Hence, the first step is
to prove that any solution $u$ of~\eqref{eq:stefan.nolocal} is the
limit of solutions $u_n$ to~\eqref{eq:approximate.problems}.

    Let $\omega$ be an open set whose closure is contained in $\R^N\times(0,\infty)$, $\omega\subset\subset\R^n\times(0,\infty)$.
    By the conservation of energy,
    $\|u_n(t)\|_{\L^1(\R^N)}=\|f\|_{\L^1(\R^N)}$.  Hence $\{u_n\}$ is uniformly bounded in $\L^1(\omega)$.
    Therefore, in order to apply Fr\'{e}chet-Kolmogorov's compactness criterium,
    it is enough to control
    $$
    I=\iint_\omega|u_n(x+h,t+s)-u_n(x,t)|\dx\dt
    $$
    for $h$ and $s$ small enough (how small not depending on $n$).

    On one hand,  thanks to the $\L^1$-contraction property,
    \begin{equation}
    \label{eq:FK.space.term}
    \begin{array}{l}
    \displaystyle\int_0^T\int_{\mathbb{R}^N} | u_n(x+h,t+s)-u_n(x,t+s)
    |\dx\dt\\[10pt]
    \displaystyle\qquad\leq \int_0^T \!\!\int_{\R^N}| f(x+h)-f(x)|\dx\dt\leq To_h(1)
    \end{array}
    \end{equation}
    as $h\to 0$ uniformly in $s$ and $n$. On the other hand, using the
    regularity in time, then Fubini's Theorem, and finally the
    sublinearity of $\Gamma_n$ and the
    $\L^1$-contraction property, we get
    \begin{equation}
    \label{eq:FK.time.term}
    \begin{array}{l}
    \displaystyle\int_0^T \int_{\mathbb{R}^N} | u_n(x,t+s)-u_n(x,t)
    |\dx\dt\\[10pt]
    \displaystyle\hskip2cm\leq \int_0^T \int_{\mathbb{R}^N} \int_t^{t+s} |\partial_t
    u_n|(x,\tau)\dtau\dx\dt\\[10pt]
    \displaystyle\hskip2cm
       = \int_0^T \int_t^{t+s} \int_{\R^N} |J*\Gamma_n(u_n) -
       \Gamma_n(u_n)|(x,\tau)\dx\dtau\dt\\[10pt]
    \displaystyle       \hskip2cm\leq \int_0^T \int_t^{t+s} (\|J\|_{\L^\infty(\R^N)}+
    1)\|\Gamma_n(u_n(\tau))\|_{\L^1(\R^N)}\dtau\dt\\[10pt]
    \displaystyle\hskip2cm
           \le (\|J\|_{\L^\infty(\R^N)}+ 1)\|f\|_{\L^1(\R^N)} sT.
    \end{array}
    \end{equation}
    Taking $T$ such that $\omega\subset \R^N\times(0,T)$, and using
    the estimates~\eqref{eq:FK.space.term} and~\eqref{eq:FK.time.term}
    we get the required control.

    Summarizing, along a subsequence
    (still noted $u_n$), $u_n\to\bar{u}$ in $\L^1_{\rm loc}(\R^N\times(0,\infty))$ for some function $\bar{u}$.
    Moreover: (i) since the sequence $\{u_n(t)\}$ is uniformly bounded in $\L^1(\R^N)$,
    we deduce from Fatou's lemma that for almost every $t>0$, $\bar
    u(t)\in\L^1(\R^N)$; (ii) using that the nonlinearities $\Gamma_n$ are
    uniformly Lipschitz, and their uniform convergence, we get that $\Gamma_n(u_n)\to\Gamma(\bar u)$ in
    $\L^1_{\rm loc}(\R^N\times(0,\infty))$; (iii) as a consequence, since $J$ is compactly supported,
    $J\ast \Gamma_n(u_n)\to J\ast \Gamma(\bar u)$ in $\L^1_{\rm loc}(\R^N\times(0,\infty))$.
    All this is enough to pass to the limit in the integrated version of~\eqref{eq:approximate.problems},
    $$
    u_n(t)=f+\int_0^t \big(J\ast
    \Gamma_n(u_n(s))-\Gamma_n(u_n(s))\big)\ds,
    $$
    for almost every $t>0$. If we extend $\bar u(t)$ to all $t>0$ by continuity, so that it belongs to the space
    $\C^1([0,\infty);\L^1(\mathbb{R}^N))$, we get that  $\bar u$ is the $\L^1$-solution
    to~\eqref{eq:stefan.nolocal} with initial data $f$, i.e., $\bar
    u=u$. As a consequence,  convergence is not restricted to a subsequence.

    Now we turn to the contraction property. Let $u_1$, $u_2$ be the $\L^1$-solutions with initial data $f_1$ and $f_2$ respectively.
    We approximate them by the above procedure, which yields sequences
    $\{u_{n,i}\}$, $i=1,2$, such that $u_{n,i}\to u_i$ in $\L^1_{\rm loc}(\R^N\times(0,\infty))$ (and hence a.e.). The approximations
    satisfy \eqref{eq:contraction.approximate}.
    Using Fatou's lemma to pass to the limit in this last inequality we get that \eqref{eq:contraction.enthalpy} holds for almost every $t\ge0$.
    Finally, since the solutions are in $\C([0,\infty);\L^1(\mathbb{R}^N))$, we deduce that this inequality holds for any $t\ge0$.
\end{proof}

\remark
Equation \eqref{eq:contraction.enthalpy} implies a comparison principle. In particular, if $f\ge0$, we conclude that $u\ge0$. Hence, $u$ is   truly a \emph{one-phase} solution.

The temperature turns out to be subcaloric.
\begin{lemma}
\label{lem:subcaloric}
    Let $f\in\L^1(\mathbb{R}^N)$ and $u$ the corresponding $\L^1$-solution. Then
    the temperature   $v=(u-1)_+$ satisfies
    $v_t\leq J\ast v -v$ in the sense of distributions and a.e. in $\R^N\times(0,\infty)$.
\end{lemma}
\begin{proof}
    Since $u\in\C^1([0,\infty);\L^1(\R^N))$,
    $$
        v_t=u_t \mathds{1}_{\{u>1\}}\quad \mbox{  a.e.}
    $$
    In the set  $\{u\leq 1\}$ we have $u_t= J\ast(u-1)_+\geq 0$   and $v_t=0$, whereas in the set $\{u>1\}$
    we have $v_t=u_t$ a.e.. In both cases we obtain $v_t\leq u_t=J\ast v -v$ a.e., and in the sense of distributions
    since these are locally integrable functions.
\end{proof}

This property allows to estimate the size of the solution in terms
of the $\L^\infty$-norm of the initial data.
\begin{lemma}
\label{lem:bound.infty}
    Let $f\in\L^1(\mathbb{R}^N)\cap\L^\infty(\mathbb{R}^N)$. Then the $\L^1$-solution $u$ of~\eqref{eq:stefan.nolocal}
    satisfies $\|u(t)\|_{\L^\infty(\R^N)}\leq \|f\|_{\L^\infty(\R^N)}$ for any $t>0$.
    Moreover, $\limsup_{t\to\infty} u(t)\leq1$ a.e.~in~$\R^N$.
\end{lemma}

\begin{proof}
    The result is obvious if $\|f\|_{\L^\infty(\R^N)}\leq 1$, since in this case $u(t)=f$ for any $t>0$. So let us assume that
    $\|f\|_{\L^\infty(\R^N)}>1$.
    Since $v$ is subcaloric and locally integrable we may use \cite[Proposition~3.1]{BrandleChasseigneFerreira2010} (with a compactly supported kernel),
    to obtain that
    $$0\leq \|v(t)\|_{\L^\infty(\R^N)}\leq\|(f-1)_+\|_{\L^\infty(\R^N)}=\|f\|_{\L^\infty(\R^N)}-1.$$
    Therefore,  $\|u(t)\|_{\L^\infty(\R^N)}\leq 1+\|v(t)\|_{\L^\infty(\R^N)}\leq\|f\|_{\L^\infty(\R^N)}$.

    We can also compare $v$ with the solution $w$ of the following problem:
    $$w_t=J\ast w -w,\quad w(0)=(f-1)_+\in\L^1(\mathbb{R}^N)\cap\L^\infty(\mathbb{R}^N).$$
    Using Theorem~\ref{app:thm:refined} in the Appendix, we obtain that the solution $v$ goes to zero asymptotically
    like $ct^{-N/2}$, so that $(u-1)_+\to0$ almost everywhere, which implies the result.
\end{proof}

\subsection{$\mathrm{BC}$ theory}

We now develop a theory in the class of continuous and bounded
functions whenever the initial data $f$ belong to that class.

\begin{definition} Let $f\in\mathrm{BC}(\mathbb{R}^N)$. The function $u$ is a
$\mathrm{BC}$-solution of~\eqref{eq:stefan.nolocal} if $u\in
\mathrm{BC}(\mathbb{R}^N\times[0,T])$ for all $T\in(0,\infty)$ and
$$
        u(x,t)=f(x)+\int_0^t \big(J\ast v(x,s)-v(x,s)\big)\ds,\quad v=(u-1)_+
$$
for all $x\in\mathbb{R}^N$ and $t\in[0,\infty)$.
\end{definition}

Notice that if $u$ is a $\mathrm{BC}$-solution, then $u_t$ is
continuous. Hence equation~\eqref{eq:stefan.nolocal} is satisfied
for all $x$ and $t$, and $u$ is a \emph{classical} solution.

\begin{theorem}\label{thm:stefan.noloc.BC}
    For any $f\in\mathrm{BC}(\mathbb{R}^N)$ there exists a unique $\mathrm{BC}$-solution of~\eqref{eq:stefan.nolocal}.
\end{theorem}
\begin{proof}
As in the case of $\L^1$-solutions, existence and uniqueness follow
from a fixed-point argument.

We start by proving existence and uniqueness in a small time
interval $[0,t_0]$. We define the operator
$\mathcal{T}:\mathrm{BC}(\mathbb{R}^N\times[0,t_0])\to\mathrm{BC}(\mathbb{R}^N\times[0,t_0])$
through
$$
  (\mathcal{T}_{f}u)(x,t)=f(x)+\int_0^t \left(J*(u-1)_+(x,s)-(u-1)_+(x,s)\right)\ds.
$$
This operator is contractive if $t_0<1/2$, which implies the local
existence and uniqueness result. Indeed, a similar computation to
that of the proof of Theorem~\ref{thm:stefan.noloc.L1} yields
$$
|\mathcal{T}_{f}\varphi-\mathcal{T}_{f}\psi|(x,t) \leq 2t_0
  \max_{\mathbb{R}^N\times[0,t_0]} |\varphi-\psi|,\qquad t\in[0,t_0].
$$
By iteration, taking as initial data
$u(x,t_0)\in\mathrm{BC}(\mathbb{R}^N)$, we obtain existence  and
uniqueness for $[0,2t_0]$ and hence for all times.
\end{proof}

The $\mathrm{BC}$-solutions depend continuously on the initial
data.
\begin{lemma}\label{lem:cont.dep.bc}
  Let $u_1$ and $u_2$ be the $\mathrm{BC}$-solutions with initial data $f_1$, $f_2\in\mathrm{BC}(\mathbb{R}^N)$ respectively .
  Then, for all $T\in(0,\infty)$
  there  exists a constant $C=C(T)$ such that
  $$
    \max_{x\in\mathbb{R}^N}|u_1-u_2 |(x,t) \leq
    C(T)\max_{\mathbb{R}^N}|f_1-f_2|,\qquad t\in[0,T].
  $$
\end{lemma}

\begin{proof} Since $u_i$ is a fixed point of the operator
$\mathcal{T}_{f_i}$, we have (see the proof of
Theorem~\ref{thm:stefan.noloc.BC})
$$
 |u_1-u_2|(x,t)
 \leq
 |f_1-f_2|(x)+2t_0\max_{\mathbb{R}^N\times[0,t_0]}|u_1-u_2|,\quad
 x\in\mathbb{R}^N, t\in[0,t_0].
$$
Taking $t_0=1/4$, we get
$$
\max_{\mathbb{R}^N\times[0,1/4]}|u_1-u_2|\leq 2 \max_{\mathbb{R}^N}
 |f_1-f_2|,
$$
from where the result follows by iteration, with a constant
$C(T)=2^{4T}$.
\end{proof}

We also have a control of the size of the solutions in terms of the
initial data. The proof is identical to the one for
$\L^1$-solutions.

\begin{lemma}\label{lem:subcaloric.linfty.bc}
Let $u$ be the $\mathrm{BC}$-solution $u$ of~\eqref{eq:stefan.nolocal} with initial data $f\in\mathrm{BC}(\R^N)$, and let $v$ be the corresponding temperature. Then, $v$ is subcaloric, and $\|u(\cdot,t)\|_{\L^\infty(\R^N)}\leq \|f\|_{\L^\infty(\R^N)}$ for any $t>0$.
\end{lemma}

\section{Free boundaries and mushy regions}
\setcounter{equation}{0}
\newcommand{\Dsupp}{\supp_{\mathcal{D}'}}

In the sequel, unless
we say explicitly  something different, we will be dealing with $\L^1$-solutions.
Since the functions we are handling
are in general not continuous in the space variable, their
positivity sets have to be considered in the distributional sense.
To be precise, for any locally integrable and nonnegative function
$g$ in $\R^N$, we can consider the distribution $T_g$ associated to
the function $g$. Then the distributional support of $g$,
$\Dsupp(g)$ is defined as the support of $T_g$:
$$
\Dsupp(g):=\R^N\setminus\mathcal{O},\quad\text{where
}\mathcal{O}\subset\R^N \text{ is the biggest open set such that
}T_g|_\mathcal{O}\equiv 0.
$$
In the case of nonnegative functions $g$, this means that
$x\in\Dsupp(g)$ if and only if
$$
\forall\varphi\in\mathrm{C}^\infty_\textrm{c}(\R^N),\
\varphi\geq0\text{ and }\varphi(x)>0\Longrightarrow \int_{\R^N}
g(y)\varphi(y)\dy>0.
$$
If $g$ is continuous, then the support of $g$ is nothing but the
usual closure of the positivity set, $\Dsupp(g)=\overline{\{g>0\}}$.

\subsection{Existence of free boundaries}

We first prove that the solution does not move  far away from the
support of $v$.

\begin{lemma}\label{lem:ut=0}
    Let $f\in\L^1_+(\R^N)$. Then,
    $$
    \Dsupp(u_t(t))\subset \Dsupp(v(t))+B_{R_J}\ \text{ for any }
    t\geq0.
    $$
\end{lemma}

\begin{proof}
Recall first that the equation holds down to $t=0$ so that we may consider here $t\geq0$ (and not only $t>0$).
Let $\varphi\in\mathrm{C}^{\infty}_{\rm c}(A^c)$, where
$A=\Dsupp(v(t))+B_{R_J}$. Notice  that the support of $J\ast
v$ (which is a continuous function) lies inside $A$, so that
$$
\int_{\R^N} (J\ast v)\varphi=0.
$$
Similarly, the supports of $v$ and $\varphi$ do not intersect, so
that
$$
\int_{\R^N} u_t\varphi=\int_{\R^N} (J\ast v)\varphi-\int_{\R^N}
v\varphi=0,
$$
which means that the support of $u_t$ is contained in $A$.
\end{proof}

As a direct consequence, we get the finite speed of propagation
property.
\begin{theorem}\label{thm:free.bndry}
    Let $f\in\L^1_+(\mathbb{R}^N)$ and compactly supported. Then, for any $t>0$, the solution
    $u(t)$ and the corresponding temperature $v(t)$ are compactly supported.
\end{theorem}

\begin{proof}\textsc{Estimate of the support of $v$. }
    Notice first that
    $$
    \big(J*(u-1)_+\big)(x,t)\leq \|J\|_{\L^\infty(\R^N)}\|(u-1)_+\|_{\L^1(\R^N)}\leq  \|J\|_{\L^\infty(\R^N)}\|f\|_{\L^1(\R^N)}:=c_0,
    $$
    where we have used the $\L^1$-contraction property of the equation for the last estimate.
    Multiplying \eqref{eq:form.int} by a nonnegative test function $\varphi\in\mathrm{C}^\infty_\textrm{c}\left((\Dsupp f)^c\right)$ and integrating in space and time we have
    $$
    \int_{\R^N} u(t)\varphi\le
    \int_0^t\int_{\R^N} \big(J*(u(t)-1)_+\big)\varphi \leq c_0\, t\,\int_{\R^N} \varphi\,.
    $$
    Taking $t_0:=1/c_0$, we get
    $\int_{\R^N} (u(t)-1)\varphi\leq 0$ for all $t\in[0,t_0]$. Using an approximation $\varphi\chi_n$ where $\chi_n\to\sign_+(u-1)$,
    we deduce that $\int_{\R^N}  (u(t)-1)_+\varphi= 0$,
    so that
    \begin{equation}
    \label{eq:estimate.support.v.small.times}
    \Dsupp (v(t))\subset \Dsupp(f), \qquad t\in[0,t_0].
    \end{equation}

    \noindent\textsc{Estimate of the support of $u$. }
    Lemma~\ref{lem:ut=0} implies then that
    $$
    \Dsupp(u_t(t))\subset \Dsupp(f)+B_{
        R_J},\qquad t\in[0,t_0].
    $$
    This means that for any $\varphi\in\mathrm{C}^\infty_{\rm c}\left(\left(\Dsupp(f)+B_{
        R_J}\right)^c\right)$ we have
    $$
        \int_{\mathbb{R}^N}  u(t)\varphi=\int_0^t\int_{\mathbb{R}^N}  u_t(t)\varphi
        =0,\qquad t\in[0,t_0],
    $$
    that is,
    \begin{equation}
    \label{eq:estimate.supp.u.above}
    \Dsupp(u(t))\subset \Dsupp(f)+B_{
        R_J},\qquad t\in[0,t_0].
    \end{equation}

   \noindent\textsc{Iteration. } Notice that $t_0$ depends on the initial data $f$ only through its $\L^1$ norm.
    Hence, since the $\L^1$ norm of the enthalpy is time invariant,  the arguments can be
    iterated to obtain the result for all times.
\end{proof}

\remark The same argument can be used for initial data which are not compactly supported, to show that some positive time will pass before the temperature becomes positive at any given point in
the complement of the support of $(f-1)_+$.

The last two results have counterparts for $\mathrm{BC}$-solutions.
\begin{theorem}\label{th:fsp.bc}
Let $f\in\mathrm{BC}_+(\R^N)$, and let $u$ be the corresponding
$\mathrm{BC}$-solution. Then:
\begin{itemize}
\item[\rm(i)]
    $u_t(x,t)=0$ for any $x\in (\supp(v(\cdot,t))+B_{R_J})^c$, $t\ge0$.

\item[\rm(ii)] If $\sup_{|x|\ge R}f(x)<1$ for some $R>0$,   then $v(\cdot,t)$ is compactly supported for
all $t>0$. If moreover $f\in\C_{\rm c}(\R^N)$, then
$u(\cdot,t)$ is also compactly supported for all $t>0$.
    \end{itemize}
\end{theorem}
\begin{proof}
(i) The proof is  similar (though even easier, since the supports
are understood in the classical sense) to the one for
$\L^1$-solutions.

\noindent (ii) Using Lemma~\ref{lem:subcaloric.linfty.bc} we get
    $$
    \big(J*(u-1)_+\big)(x,t)\leq \|J\|_{\L^1(\R^N)}\|(u-1)_+\|_{\L^\infty(\R^N)}\leq
    \|f\|_{\L^\infty(\R^N)}.
    $$
Therefore, from~\eqref{eq:form.int} we have
\begin{equation}
  \label{eq:bound.for.u}
u(x,t)\le f(x)+t\|f\|_{\L^\infty(\R^N)}\le\sup_{|x|\ge R} f(x)
+t\|f\|_{\L^\infty(\R^N)},\qquad |x|\ge R.
\end{equation}
Thus, for all $|x|\ge R$ and $t\le(1-\sup_{|x|\ge
R}f(x))/(2\|f\|_{\L^\infty(\R^N)})$ we have $u(x,t)<1$, and hence
$v(x,t)=0$. Then, by (i), $u(x,t)=f(x)$ for all $|x|\ge R+R_J$ and
$t=(1-\sup_{|x|\ge R}f(x))/(2\|f\|_{\L^\infty(\R^N)})$. We now
proceed by iteration.
\end{proof}

\subsection{Equivalence of formulations}
As a corollary of the control of the supports, we will prove that if
the initial data are in $\L^1_+(\mathbb{R}^N)\cap \C_0(\mathbb{R}^N)$, then the $\L^1$-solution is in fact continuous.

\begin{proposition}\label{prop:eq.cs}
 Let $f\in\L^1_+(\mathbb{R}^N)\cap\C_{0}(\mathbb{R}^N)$.  The corresponding
$\L^1$-solution is continuous.
\end{proposition}

\begin{proof}
We start by considering the case where $f$ is continuous, nonnegative and
compactly supported. Since a $\mathrm{BC}$-solution with a continuous and compactly supported initial data stays
compactly supported in space for all times, it is also integrable in space for all times. Moreover, $u\in\C\big([0,T];\L^1(\mathbb{R}^N)\big)$.
Hence, it coincides with the $\L^1$-solution with the same initial data.

We now turn to the general case, that will be dealt with by
approximation: let $\{f_n\}$ be a sequence of continuous and compactly supported
functions such that
$$
 \|f_n - f\|_{\L^\infty(\R^N)} <\frac1n,\qquad \|f_n-f\|_{\L^1(\R^N)} <\frac1n.
$$
Let $u_n^1$, $u^1$ be the $\L^1$-solutions with initial data
respectively $f_n$ and $f$, and $u_n^c$, $u^c$ the corresponding
$\mathrm{BC}$-solutions. We know that $u_n^1=u_n^c$.
Now, using the $\L^1$-contraction property for
$\L^1$-solutions, we have that
$\|u_n^1-u^1\|_{\L^1(\R^N\times[0,T])}\to0$ for any
$T\in[0,\infty)$. Moreover, by Lemma~\ref{lem:cont.dep.bc}, $\|u_n^1-u^c\|_{\L^\infty(\R^N\times[0,T])}\to0$. Hence
the result.
\end{proof}

We can now use an approximation argument to prove a comparison principle for general initial data in $ \mathrm{BC}_+(\mathbb{R}^N)$. This can in turn be used to show that initial data in that class yield \emph{one-phase} solutions (satisfying $u\ge0$).

\begin{corollary}\label{cor:comparison.BC} Let $f_1,f_2\in\mathrm{BC}_+(\mathbb{R}^N)$, and $u_1$, $u_2$ the corresponding  $\mathrm{BC}$-solutions. If $f_1\leq f_2$ then $u_1\leq u_2$.
\end{corollary}

\begin{proof} Let $\{f_{1,n}\}$, $\{f_{2,n}\}$ be sequences of nonnegative, continuous and compactly supported functions such that $f_{1,n}\to f_1$, $f_{n,2}\to f_2$ uniformly, and  $f_{1,n}\leq f_{2,n}$. Since $f_{1,n},f_{2,n}\in \L^1(\mathbb{R}^N)\cap \C_0(\mathbb{R}^N)$, the corresponding $\mathrm{BC}$-solutions $u_{1,n}$, $u_{2,n}$, are also $\L^1$-solutions. Hence, the comparison principle for $\L^1$-solutions,  Corollary~\ref{cor:contraction.stefan.noloc}, yields $u_{1,n}\leq u_{2,n}$.  Passing to the limit in $\mathrm{BC}$, and using the continuous dependence of $\mathrm{BC}$-solutions on the initial data, Lemma~\ref{lem:cont.dep.bc}, we conclude
 that $u_1\le u_2$.
\end{proof}

\subsection{Retention for $u$ and $v$}

We next prove that the supports of both $u$ and $v$ are
nondecreasing. We denote this property as retention.

We start by considering the case of $\mathrm{BC}$-solutions.
\begin{proposition}
    \label{prop:retention.bc}
    Let $f\in\mathrm{BC}_+(\mathbb{R}^N)$,
    and let $u$ be the $\mathrm{BC}$-solution to problem~\eqref{eq:stefan.nolocal}.
    Then $\partial_tu\ge -u$ and $\partial_tv\ge-v$ for all $t\ge0$. In
    particular, $u$ and $v$ have the retention property.
\end{proposition}

\begin{proof}
    We have
    $$
    \partial_tu=J* (u-1)_+-(u-1)_+\geq -u,
    $$
    which, after integration, yields
    \begin{equation}
    \label{eq:integrated.retention}
    u(x, t)\geq u(x,s)\e^{-(t-s)},\qquad t\geq s.
    \end{equation}
    This implies retention for $u$.

    Concerning $v=(u-1)_+$, we have
    $$
    \partial_t(u-1)_+ = \partial_t u\cdot \mathds{1}_{\{u>1\}} \geq
    -(u-1)_+,
    $$
    that is, $\partial_tv\geq -v$, from where retention follows.
\end{proof}

For $\L^1$-solutions we also have retention for both $u$ and $v$. In
this case the supports have to be understood in the distributional
sense.

\begin{proposition}
    \label{prop:retention.l1}
    Let $f\in\L^1_+(\mathbb{R}^N)$, and let $u$ be the $\L^1$-solution to problem~\eqref{eq:stefan.nolocal}.
    Then $u$ and the corresponding $v$ have the retention property.
\end{proposition}
\begin{proof}
Let $\{f_n\}$ be a sequence of functions in $\mathrm{C}^\infty_{\rm
c}(\R^N)$ such that $\|f_n-f\|_{\L^1(\R^N)}\to0$. Let $u_n$ be the
$\L^1$-solution (which coincides with the $\mathrm{BC}$-solution)
with initial data $f_n$. Thanks to the $\L^1$-contraction property,
 $\|u_n(t)-u(t)\|_{\L^1(\R^N)}\to0$ for all $t>0$. Moreover, since the temperature $v$ is a Lipschitz function of $u$, we also have
 $\|v_n(t)-v(t)\|_{\L^1(\R^N)}\to0$.

Let $\varphi$ be any nonnegative function, $\varphi\in\mathrm{C}^\infty_{\rm
c}(\R^N)$. Multiplying \eqref{eq:integrated.retention} (with $u_n$
instead of $u$) by $\varphi$, integrating and letting $n\to\infty$,
we get
    $$
    \int_{\mathbb{R}^N}  u(t)\varphi \geq \e^{-(t-s)}\int_{\mathbb{R}^N}  u(s)\varphi,\qquad t\ge
    s,
    $$
from where retention for $u$ in the distributional sense is
immediate. The argument for $v$ is identical.
\end{proof}

\subsection{Localization of the supports}

Our next aim is to prove, in the case of a nonincreasing kernel $J$,
that for a wide class of initial data the supports of both $u$ and $v$ are \emph{localized}: they are
contained in a ball of fixed radius for all times.

The result will follow from comparison with solutions with initial
data in $\C_0(\R^N)$ that are radial and strictly
decreasing in the radial variable. Such solutions are continuous and
radial (the latter fact comes from the uniqueness of the solutions
and invariancy under rotations of the equation). Moreover, for any time $t\ge0$, the water zone $\{v(t)>0\}$ is compactly supported (see Theorem~\ref{th:fsp.bc}).
The main technical difficulty stems from the fact that we are not able to prove that these
solutions are decreasing in the radial variable for all times.

\begin{lemma}\label{lem:radial.data}
Let $J$ be nonincreasing in the radial variable and
$f\in\C_0(\R^N)$  nonnegative, radial, and strictly decreasing in the
radial variable. Then the support of $v(t)$ is a ball of radius $r(t)$
for every $t\ge0$ and the function $r$ is continuous on $[0,\infty)$.
\end{lemma}

\begin{proof}
    For $t>0$, let $r(t):=\inf\{r>0:\supp(v(\cdot,t))\subset
    B_{r(t)}\}$. Thanks to Theorem~\ref{th:fsp.bc} this quantity is well
    defined. By the retention property for $v$
   (Proposition~\ref{prop:retention.l1}), the function $r$ is nondecreasing.
    Notice that a priori $\supp(v(t))$  could be strictly contained in $B_{r(t)}$,
    though we will prove that this is not the case.

   \noindent{\sc Continuity of $r$. }   Assume for contradiction that
 $r(t_0^-)<r(t_0^+)$  at some time $t_0>0$. For any
    $x$ such that $|x|>r(t_0^-)$ we have $v(x,t)=0$ for all
    $t\le t_0$. If moreover $|x|=r(t_0^+)$, the continuity of $u$ yields $u(x,t_0)=1$.

    Let
    $ x_a=(a,0,\dots,0)$, $x_b=(b,0,\dots,0)$, with
    $r(t_0^-)<a<b=r(t_0^+)$.
    We consider $w(t):=u(x_a,t)-u(x_b,t)$.
    For any
    $t\le t_0$, we have
    $$
    \begin{aligned}
        w'(t) & = \left((J\ast v)(x_a,t)- v(x_a,t)\right)-\left((J\ast v)(x_b,t) - v(x_b,t)\right)\\
            & = (J\ast v)(x_a,t)-(J\ast v)(x_b,t)\\
            & = \int_{|y|<r(t_0^-)} v(y,t)\big(J(x_a-y)-J(x_b-y)\big)\dy.
    \end{aligned}
    $$
    Since $|x_a-y|<|x_b-y|$, for all $|y|\le r(t_0^-) $, then $J(x_a-y)\ge J(x_b-y)$ in this region, because $J$ is radially noincreasing.
    Thus we obtain $w'\geq 0$. Using that
    $w(0)=f(x_a)-f(x_b)>0$,
    we obtain $w(t_0)>0$, so that $u(x_a,t_0)>u(x_b,t_0)=1$. This is a contradiction,
    since  $v(x_a,t_0)=0$.

Continuity at $t=0$ is easier. On one hand, from the retention property we have $r(0^+)\geq r(0)$. On the other hand,  since $f$ is strictly decreasing, we have
$f(x)<1$ if and only if $|x|>r(0)$. But, by the continuity of $u$, $f(x)=1$ if $|x|=r(0^+)$. Therefore, $r(0^+)$ cannot be strictly greater that $r(0)$. We end up with $r(0)=r(0^+)$.

    \noindent{\sc Conectedness. } We now prove that $\supp(v(\cdot,t))$
    is connected for all positive times. Assume, on the contrary, that
    there exists some $t_*>0$ such that $\supp(v(\cdot,t_*))$ is
    disconnected. Hence, there are values $r(0)<a<b<r(t_*)$ such that
    $v(x,t_*)=0$ if $a\leq|x|\leq b$ and $v(x,t_*)>0$ if $b<|x|<b+\delta$ for some $\delta>0$.
    The retention property for $v$ implies that
    $v(x,t)=0$ for $a\le|x|\le b$, $0\leq t<t_*$.

    Let $t_d\in(0,t_*)$ be the time when the disconnected region outside the ball $B_b$ appears,
    $$
    t_d:=\sup\{t>0:v(x,t)=0\text{ for }|x|\ge b\}=\sup\{t>0:v(x,t)=0\text{ for }|x|\ge
    a\}.
    $$
    Obviously, $r(t_d)\le a$ and, on the other hand,
    $r(t_d^+)\geq b>a$, a contradiction with the continuity of $r$.
\end{proof}

\begin{lemma}
\label{lem:cont.supp.v}Let $J$ be nonincreasing in the radial
variable. If $0\le f\le g$  a.e.~for some  $g\in\L^1_+(\R^N)\cap\C_0(\R^N)$
radial and strictly decreasing in the radial variable, then there
exists some $R$,  depending only on $g$, such that
$\supp(v(t))\subset B_R$ for all $t\ge0$.
\end{lemma}

\begin{proof} Let $u_g$ be the $\L^1$-solution with initial datum $g$ and $v_g=(u_g-1)_+$.
    By comparison,  $\supp(v(t))\subset\supp(v_g(t))$.
    Lemma \ref{lem:radial.data} implies that  $\supp(v_g(t))=B_{r_g(t)}$. The radius $r_g(t)$ can be
    estimated using the conservation of mass,
    $$
    \int_{\R^N}g= \int_{\R^N}u_g(t)\geq\int_{\{u_g(t)>1\}}1
    =\big| \supp(v_g(t))\big|=\omega_Nr_g(t)^N,
    $$
    where $\omega_N$ is the volume of the unit sphere in $\R^N$. This implies that
    $$
    \supp({v}(t))\subset\supp(v_g(t))\subset B_R, \qquad
    R=\left(\int_{\R^N}g/\omega_N\right)^{1/N}.
    $$
\end{proof}

As a corollary  of the localization of the support of the
temperature, we obtain that $\|v(t)\|_{\L^1(\R^N)}$ tends to zero as $t\to\infty$ with an
exponential rate.

\begin{corollary}
\label{corollary:exponential.decay}Let $J$ and $f$ satisfy the hypotheses of Lemma~{\rm\ref{lem:cont.supp.v}}. Then there are constants $C,k>0$ such that $
\|v(t)\|_{\L^1(\R^N)}\le C \e^{-kt}$ for all $t\ge0$.
\end{corollary}

\begin{proof} Let $R$ be such that $\supp(v(t))\subset B_R$ for all
times, and $V$ the $\L^1$-solution to the nonlocal heat equation in $B_R$,
$$
\left\{
\begin{array}{ll}
V_t(x,t)=J* V(x,t)-V(x,t), & x\in B_R\ t>0,
\\
V(x,t)=0, & x\notin B_R,\ t>0,
\\
 V(x,0)=(f-1)_+(x), &
x\in B_R.
\end{array}
\right.
$$
Since $f$ is bounded, $(f-1)_+\in\L^2(B_R)$. Hence
\cite[Theorem 2]{ChasseigneChavesRossi2006},  $\|V(t)\|_{\L^2(B_R)}$  decays
exponentially in time.

As $v$ is subcaloric, Lemma~\ref{lem:subcaloric},  $v\le V$ in
$B_R\times(0,\infty)$. This implies
$$
\int_{\R^N}v(t)\leq (\omega_NR^N)^{1/2}\|V(t)\|_{\L^2(B_R)}\leq
C\e^{-kt}.
$$
\end{proof}

For general integrable data we are only able to obtain a power-like
decay rate.
\begin{corollary}
\label{cor:decay.rate.integrable} Let $f\in\L^1(\R^N)$. Then
$\|v(t)\|_{\L^1(\R^N)}=O(t^{-N/2})$.
\end{corollary}
\begin{proof}
Since $v$ is subcaloric and nonnegative, it is enough to compare it from above with the solution, $V$, to the
non-local heat equation \eqref{app:eq:nonloc.heat} with the same initial data.
Thus, using the representation formula \eqref{eq:representation.formula.nlhe} for solutions to \eqref{app:eq:nonloc.heat} (see the appendix), we get
$$
\int_{\R^N} v(t)=\int_{\{v(t)>0\}} V(t)\le e^{-t}\|(f-1)_+\|_{\L^1(\mathbb{R}^N)}+
\int_{\{v(t)>0\}} \omega(t)*(f-1)_+,
$$
where $\omega$ is the regular part of the fundamental solution to \eqref{app:eq:nonloc.heat}.
Then we notice that the measure of the support of $v(t)$ is uniformly controlled,
$$
|\{v(t)>0\}|=|\{u(t)\geq 1\}|\leq \int_{\{u(t)\geq 1\}}u(t)\leq \|f\|_{\L^1(\mathbb{R}^N)}.
$$
Thus, using that $\|\omega(t)\|_{\L^\infty(\mathbb{R}^N)}\leq Ct^{-N/2}$ (see \cite{IgnatRossi}),
we obtain
$$
\int_{\R^N} v(t)\leq \e^{-t}\|(f-1)_+\|_{\L^1(\mathbb{R}^N)} + Ct^{-N/2}\|f\|_{\L^1(\mathbb{R}^N)}\|(f-1)_+\|_{\L^1(\mathbb{R}^N)}=O(t^{-N/2}).
$$
\end{proof}

If the initial data are bounded and compactly supported, we can
obtain quantitative estimates for the supports of $u$ and $v$. These
estimates are sharp, as can be checked by considering indicator
initial data.

\begin{lemma}Let $J$ be nonincreasing in the radial variable and
$f$ nonnegative, bounded and compactly supported, contained in the ball of radius
$R_f$. Then
$$
\supp(v(t))\subset B_{R_v}, \quad \supp(u(t))\subset
B_{\max\{R_f, R_v+R_J\}},\qquad
R_v=\|f\|_{\L^\infty(\R^N)}^{1/N}R_f.
$$
\end{lemma}
\begin{proof}
To obtain the estimate for the support of $v$ we use
Lemma~\ref{lem:cont.supp.v} with functions $g_n$ approximating
$\mathds{1}_{B_{R_f}}\|f\|_{\L^\infty(\R^N)}$ from above, and then
pass to the limit in $n$. The estimate for the support of $u$ then
follows from Lemma~\ref{lem:ut=0}.
\end{proof}

\remark
    A better result should hold with a radius $R_v$ depending on the mass of the initial data above level one, instead of the $\L^\infty$-norm of $f$.

\subsection{Creation of mushy regions}
\label{subsec.mushy}

    Since the $\L^1$-solutions to our problem are not necessarily continuous,
    we need a distributional definition of the mushy region.
\begin{definition}The mushy region at time $t\geq0$ of a nonnegative ($\L^1$- or $\mathrm{BC}$-) solution $u$ to~\eqref{eq:stefan.nolocal} is
    $$
        \mathcal{M}(t):=\mathrm{Int}\Big(\Dsupp(u(t))\setminus\Dsupp((u(t)-1)_+)\Big).
    $$
\end{definition}

\remark
When $u$ is continuous, $\mathcal{M}(t)=\{0<u(x,t)<1\}$.

Here comes one of the main features of our model:  it allows the
creation of mushy regions.
    \begin{theorem}
    \label{thm:mushy.creation} Let $f\in\L^1_+(\R^N)$ be a nontrivial initial data such
    that $\Dsupp(f)=\Dsupp((f-1)_+)$.
    Then,
     $$
     \mathcal{M}(t)=\big\{0<\dist\big(x,\Dsupp(f))<R_J\big\},\quad
    t\in[0,t_0],\quad t_0=\frac{1}{\|J\|_{\L^\infty(\R^N)}\|f\|_{\L^1(\R^N)}}.
    $$
\end{theorem}
\begin{proof} We first observe that the assumptions on the initial data imply that there are no mushy regions initially, $\mathcal{M}(0)=\emptyset$. Moreover,
\begin{equation}
\label{eq:equality.supports.short.times}
\Dsupp(v(t))=\Dsupp(f),\qquad t\in[0,t_0].
\end{equation}
The upper inclusion is just
\eqref{eq:estimate.support.v.small.times}, and the lower one follows
from the retention property and the equality of the supports of $f$
and $(f-1)_+$.

The inclusion $
\mathcal{M}(t)\subset\big\{0<\dist\big(x,\Dsupp(f))<R_J\big\}$ is an
immediate consequence of equations~\eqref{eq:estimate.supp.u.above}
and~\eqref{eq:equality.supports.short.times}.

Let us turn then to the other inclusion. Let $\varphi$ be a
nonnegative and nontrivial test function compactly supported in
$\{0<\dist(x,\Dsupp(f)\big)<R_J\}$.
Using~\eqref{eq:estimate.support.v.small.times}, we have
$$
\begin{aligned}
        \int_{\R^N}u(t)\varphi &= \int_{\R^N}f\varphi +\int_0^t\int_{\R^N} J\ast v(t)\varphi -
            \int_0^t\int_{\R^N}v(t)\varphi\\
                 &=  \int_0^t\int_{\R^N} J\ast v(t)\varphi,\qquad
                 t\in[0,t_0].
\end{aligned}
$$
Since  $\supp(J\ast v(t))=
\Dsupp((v(t))+B_{R_J}=\Dsupp(f)+B_{R_J}$, we conclude that
    $$
    \int_{\R^N} u(t)\varphi=\int_{\R^N} J\ast v(t)\varphi>0, \quad t\in[0,t_0].
    $$
    In other words,
    $$
    \{0<\dist(x,\Dsupp(f))\big)<R_J\} \subset \Dsupp(u(t)), \qquad
    t\in[0,t_0],
    $$
    which combined with~\eqref{eq:equality.supports.short.times}
    gives the required inclusion.
\end{proof}

\subsection{Emergence of disconnected water regions}
\label{subsect:disconnected}

Another interesting feature of our model  is that
disconnected components of water may appear suddenly at a positive
distance from the already existing water components. This is another
example of a phenomenon that occurs for the non-local model but not
for the local one. The reason is that, contrary to the local model,
problem~\eqref{eq:stefan.nolocal} allows middle-range interactions,
(up to  a distance $R_J$).

Let us now construct examples exhibiting this phenomenon.  We will
keep things simple, so that the underlying mechanism is better
understood. But the result can be easily generalized to more complex
situations.

Let us assume that $f$ is a bounded and continuous initial data with
three well differentiated zones:
\begin{equation}
\label{eq:hipotheses.zones}
\begin{cases} \text{ a \lq\lq warm'' water
zone, $\mathcal{W}$, where the enthalpy is above~1;}\\
\text{ a low-enthalpy ice zone, $\mathcal{I}_L$, where $f$
is clearly below~1;}\\
 \text{ a \lq\lq high''-enthalpy ice zone, $\mathcal{I}_H$, where $f$ is close to, but below
 level~1.}
\end{cases}
\end{equation}
We shall see that if the enthalpy in $\mathcal{I}_H$ is close
enough to 1 and this zone is not too far from $\mathcal{W}$, then
$\mathcal{I}_H$ melts  before the low-enthalpy zone does. Therefore,
if $\mathcal{I}_L$ \lq\lq separates'' $\mathcal{W}$ and
$\mathcal{I}_H$ initially,  a new disconnected component will emerge
in the water zone.
It is enough to prove it for $\mathcal{I}_H=\{x\}$,
the general case resulting from this.
\begin{definition}
    A set $\mathcal{S}\subset\R^N$ separates $\mathcal{A}$ and $\mathcal{B}$ if:
    \begin{itemize}
    \item[(i)] $\mathcal{S}^c$ has at least two open connected components;
    \item[(ii)] $\mathcal{A}$ and $\mathcal{B}$ lie in two different connected components of $\mathcal{S}^c$.
    \end{itemize}
\end{definition}
\begin{theorem} Let $J>0$  in $B_{R_J}$, $x\in\R^N$ and consider two non-empty sets
$\mathcal{W},\mathcal{I}_L\subset\R^N$ such that
\begin{itemize}
\item $\mathcal{W}$ is open;
\item $\dist(x,\mathcal{W})<R_J$;
\item $\mathcal{I}_L$ separates $\overline{\mathcal{W}}$ and $\{x\}$.
\end{itemize}
Let $f\in\mathrm{BC}_+(\R^N)$ such that
\begin{itemize}
\item $f>1$ in $\mathcal{W}$, $f<1$ in $\overline{\mathcal{W}}^c$;
\item $0\le f\le 1-\eta$ in $\mathcal{I}_L$ for some fixed
$\eta\in(0,1)$;
\end{itemize}
Then there exists $\eps\in(0,\eta)$ such that if \ $1-\eps<f(x)<1$,
then in a finite time there appears in the water zone a new connected water component $\mathcal{C}_{x}$
containing $x$.
\end{theorem}

Figures \ref{fig.disconnect1} and \ref{fig.disconnect2} illustrate the phenomenon,
the exact meaning of notations being defined within the proof that follows.

\begin{center}
  \begin{figure}[ht!]
    \includegraphics[width=10cm]{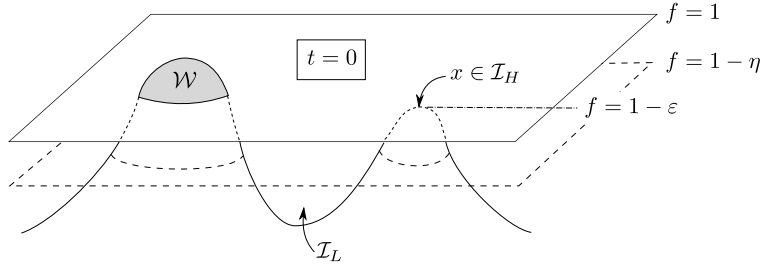}
    \caption{High-enthalpy zone near melting.}
    \label{fig.disconnect1}
  \end{figure}
\end{center}

\begin{center}
  \begin{figure}[ht!]
    \includegraphics[width=10cm]{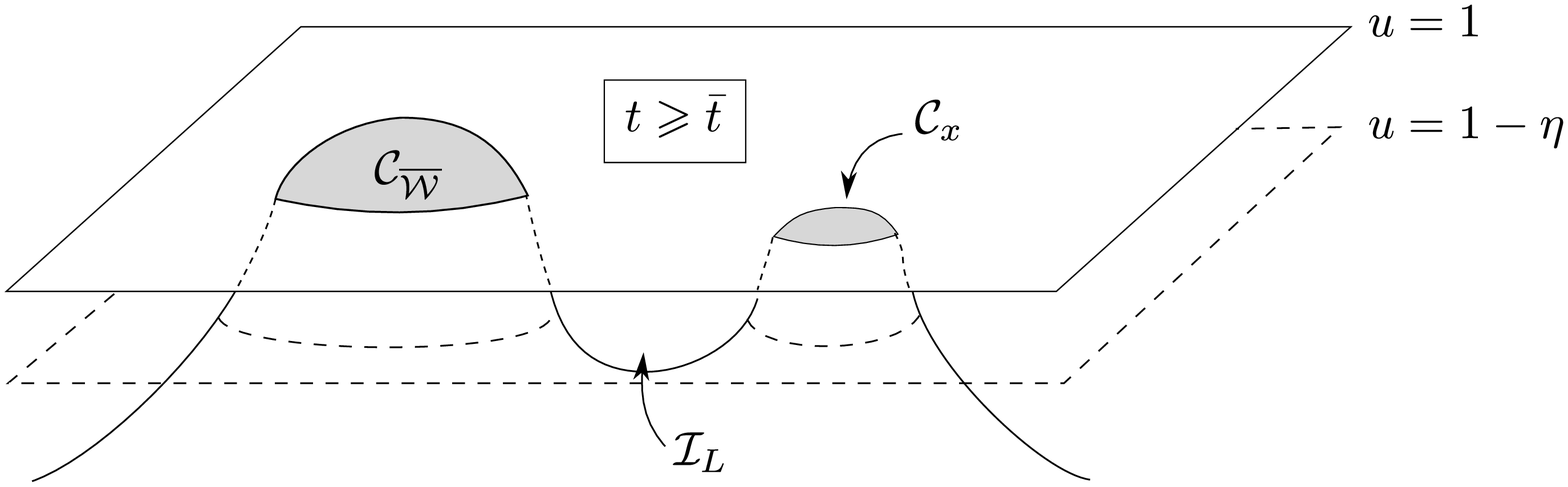}
    \caption{New disconnected component in water zone.}
        \label{fig.disconnect2}
  \end{figure}
\end{center}

\begin{proof} We proceed in three steps as follows.\smallskip

\noindent\textsc{Step 1 --}
\textit{For $0\leq t<\eta/\|f\|_{\L^\infty(\R^N)}$, $\mathcal{I}_L\cap\{v(t)>0\}=\emptyset$.}\smallskip

Let $y\in\mathcal{I}_L$ and $t<\eta/\|f\|_{\L^\infty(\R^N)}$. Using~\eqref{eq:bound.for.u},  we have $u(y,t)< 1$. Thus, such an $y$
remains in the ice zone.

\noindent\textsc{Step 2 --} \textit{Given $0<\bar t<\eta/\|f\|_{\L^\infty(\R^N)}$,
there exists $\eps_0>0$ such that for all $\eps\in(0,\eps_0)$, we have $x\in\{v(\bar t)>0\}$.} \smallskip

Assume on the contrary that $v(x,\bar t)=0$ for all $\eps>0$.
By our assumptions, since $\dist(x,\{f>1\})<R_J$, there exists $0<\rho<R_J$ and
$y_0\in\{f>1\}$ such that $\dist(x,y_0)<\rho$. Taking $\delta=(f(y_0)-1)/2>0$, we have
$f(y_0)>1+\delta$ so that the set
$$f^{(-1)}\big((1+\delta,+\infty)\big)=\{f>1+\delta\}$$
is open and contains $y_0$. Similarly, the ball $B_\rho(x)$ is also open and contains $y_0$
so that the intersection $\{f>1+\delta\}\cap B_\rho(x)$
contains at least a ball $B_\tau(y_0)$ centered at $y_0$ with a positive radius $\tau>0$.
This has two consequences that we use below: first, $\{f>1+\delta\}\cap B_\rho(x)$
has a positive Lebesgue measure; second, $x-B_\tau(y_0)\subset B_{R_J}$,
so that by assumption on $J$ (up to taking a $\tau'<\tau$),
$J(x-y)$ is uniformly bounded away from zero on $B_\tau(y_0)$.

Then, using the retention property for $v$,
Proposition~\ref{prop:retention.bc}, we have
$$
u(x,\bar t)>1-\eps + \int_0^{\bar t} (J\ast v)(x,s)\ds\ge 1-\eps +
\e^{-\bar t}\int_0^{\bar t} (J\ast v)(x,0)\ds.
$$
We estimate the integral as follows:
$$
\begin{array}{rcl}
(J\ast v)(x,0)&\geq&\displaystyle\int_{\{f>1+\delta\}\cap
B_{\rho}(x)}J(x-y)\delta\dy
\\[10pt]
&\geq& \displaystyle\delta\cdot\min_{B_\tau(y_0)}J(x-y)\cdot\big|\{f>1+\delta\}\cap B_\rho(x)\big|=C>0,
\end{array}
$$
with $C$ independent of $\eps$.
Hence, if $\eps>0$ is small enough we get
$$
u(x,\bar t)>1-\eps + \bar t\,\e^{-\bar t}C>1,
$$
which is a contradiction.

\noindent\textsc{Step 3 --} \textit{For $\bar t$ and $\eps$ as above, there is a new connected component
containing $x$ in the water zone.}\smallskip

Since $x$ and $\overline{\mathcal{W}}$ are separated by $\mathcal{I}_L$, there exist two open sets
$\mathcal{O}_1,\mathcal{O}_2$ such that $\mathcal{O}_1\cup\mathcal{O}_2\subset(\mathcal{I}_L)^c$,
$\mathcal{O}_1\cap\mathcal{O}_2=\emptyset$ and
$$\{x\}\in\mathcal{O}_1,\quad \overline{\mathcal{W}}\subset\mathcal{O}_2.
$$
Thus, initially all the water zone is contained in $\mathcal{O}_2$ while $x$ belongs to $\mathcal{O}_1$.

Now, by the retention property, we know that at time $\bar t$ the water zone still contains $\overline{\mathcal{W}}$.
But at time $\bar t$ the water zone also contains $x$. Hence we are in the following situation:
$$
\{v(\bar t)>0\}\supset \{x\}\cup\overline{\mathcal{W}}, \quad \{v(\bar t)>0\}\subset (\mathcal{I}_L)^c.
$$

At time $\bar t$, let us denote by $\mathcal{C}_x$ the connected water component containing $x$ and by
$\mathcal{C}_{\overline{\mathcal{W}}}$ the one containing $\overline{\mathcal{W}}$, which are both non-empty.
Since $\{v(\bar t)>0\}\subset (\mathcal{I}_L)^c=\mathcal{O}_1\cup\mathcal{O}_2$, the union being disjoint,
it follows that necessarily
$\mathcal{C}_x\subset\mathcal{O}_1,\ \mathcal{C}_{\overline{\mathcal{W}}}\subset\mathcal{O}_2$.
Hence we deduce that
$$\mathcal{C}_x\cap \mathcal{C}_{\overline{\mathcal{W}}}=\emptyset.$$
In other words, at time $\bar t$, a new component has appeared in $\mathcal{O}_1$ which was not present initially,
and it is even disconnected from all the water zones in $\mathcal{O}_2$.
\end{proof}

\remark
A similar phenomenon takes place for
non-continuous, integrable data. This can be proved either by using
$\L^1$-theory techniques, or   by approximation from above and from
below with continuous data.

\section{The local Stefan problem as a limit in the macroscopic scale}
\label{sect:limit.epsilon} \setcounter{equation}{0}

If the support of the kernel is shrunk to a point through a suitable
rescaling, we recover the local model.  For the case of a bounded
domain with Neumann boundary data, such convergence  was already
considered in \cite{AndreuMazonRossiToledoBook2010} in the abstract
setting of semi-group theory. We will give here an alternative, more
direct proof, adapted to our problem. In addition, we will prove
that mushy regions disappear in the limit.

Given a fixed initial datum $f\in\L^1(\R^N)\cap\L^\infty(\R^N)$, we
consider the problem
\begin{equation}
\label{eq:stefan.eps}
\partial_t u^\eps=\frac{1}{\eps^2}(J_\eps\ast
v^\eps - v^\eps),\qquad v^\eps=(u^\eps-1)_+,\qquad u^\eps(\cdot,0)=f,
\end{equation}
where $J_\eps=\eps^{-N}J(\cdot/\eps)$. Since $J_\eps$ is a unit mass
kernel, compactly supported in the ball $B_{\eps R_J}$, the various
properties of solutions of \eqref{eq:stefan.nolocal} that we derived
in the previous sections are still  valid for solutions of
\eqref{eq:stefan.eps}. This latter problem admits a weak
formulation, which will show to be quite convenient when passing to the
limit: for any test function $\phi\in \C^\infty_{\rm
c}(\R^N\times[0,\infty))$ we have
\begin{equation}\label{eq:weak.stefan.eps}
\int_{\R^N} u^\eps(t)\phi (t)=\int_{\R^N} f\phi(0)+
     \int_0^t\int_{\R^N}(\partial_t\phi) u^\eps+\frac{1}{\eps^2}\int_0^t\int_{\R^N}(J_\eps\ast\phi-\phi)v^\eps.
\end{equation}
This follows from Fubini's theorem, since $J_\eps$ is symmetric.

\begin{lemma}
\label{lem:precompactness}Let $f\in\L^1(\mathbb{R}^N)\cap \L^\infty(\R^N)$. The family $\{u^\eps\}$ is relatively compact
   in $\L_{\rm loc}^1(\R^N\times(0,\infty))$.
\end{lemma}
\begin{proof} We first prove the relative compactness of $\{u^\eps(t)\}$ in $\L^1_{\rm loc}(\R^N)$ for all $t>0$ by means of  Frechet-Kolmogorov's compactness criterium. To this aim we use the $\L^1$-contraction property to show that: (i) the functions $u^\eps(t)$ are uniformly bounded in $\L^1(\R^N)$; and (ii)  for any compact set $\omega\subset\R^N$ $$\int_\omega|u^\eps(x+h,t)-u^\eps(x,t)|\dx\leq \int_{\R^N}|f(x+h)-f(x)|\dx=o_h(1),$$
   where $o_h(1)$ tends to 0 as  $h\to0$ independently of $\eps$. Hence, along a subsequence,
   $u^\eps(t)\to u(t)$ in $\L^1_{\rm loc}(\R^N)$, for some function $u(t)\in\L^1(\R^N)$.
   We infer that $u^\eps(x,t)$ converges for almost every $(x,t)$.

   Since $\|u^\eps(t)\|_\infty\leq\|f\|_\infty$, we may now use the dominated convergence theorem to prove that
   $u^\eps$ converges to $u$ in $\L^1_{\rm loc}(\R^N\times[0,\infty))$.
\end{proof}

This compactness result gives convergence along subsequences. The
possible limit functions turn out to be weak solutions to the local
Stefan problem
\begin{equation}
\label{eq:problem.local.stefan.momentum}
\partial_tu=\frac{m_2}{2}\Delta(u-1)_+,\quad u(\cdot,0)=f,
\end{equation}
where $m_2:=\int_{\R^N}|z|^2J(z)\dz$ is the second-order momentum of
the kernel $J$, which is finite, since $J$ is compactly supported.
Since this problem has a unique weak
solution~\cite{AndreucciKorten1993}, convergence is not restricted
to subsequences.

\begin{theorem}
 Let  $f\in\L^1(\mathbb{R}^N)\cap \L^\infty(\R^N)$. The sequence  $\{u^\eps\}$ of solutions to
\eqref{eq:stefan.eps} converges  as $\eps\to0$  in $\L^1_{\rm
loc}(\mathbb{R}^N\times[0,\infty))$ to the unique weak solution
    of~\eqref{eq:problem.local.stefan.momentum}.
\end{theorem}

\begin{proof}
    Along a subsequence,  $\{u^\eps\}$ converges strongly
    in $\L^1_{\rm loc}(\mathbb{R}^N\times[0,\infty))$  to some function $u$, see Lemma~\ref{lem:precompactness}.
    We also have convergence for $\{v^\eps\}$ along some subsequence  to $v=(u-1)_+$.
   Since the $\L^1$-norms of solutions  to
\eqref{eq:stefan.eps} do not increase with time,
    see Corollary~\ref{cor:contraction.stefan.noloc}, the sequence
    $\{u_\eps\}$ is uniformly bounded in $\L^\infty((0,\infty);\L^1(\R^N))$.

If we perform a Taylor expansion and use the symmetry of $J$, we get
    $$
    \begin{array}{rcl}
    \displaystyle\frac{1}{\eps^2}\Big(J_\eps\ast\phi-\phi\Big)&=&\displaystyle\frac{1}{\eps^{2+N}}\int_{\R^N}J\Big(\frac{x-y}{\eps}\Big)
    \Big(\phi(x)-\phi(y)\Big)\dy\\[8pt]
    &=&\displaystyle\frac{1}{\eps^{2}}\int_{\R^N}J(z)\Big(\phi(x)-\phi(x-\eps
    z)\Big)\dz
    \\[8pt]
&=&\displaystyle\frac{m_2}{2}\Delta\phi+o(1)\qquad \text{as
}\eps\to0^+
    \end{array}
    $$
uniformly in $(x,t)$ (recall that $\phi$ is compactly supported and
smooth).
    Hence, passing to the limit in~\eqref{eq:weak.stefan.eps},  we get that $u$ is the unique weak solution to the local
    problem:
    for any test-function $\phi\in \C^\infty_{\rm c}(\R^N\times[0,\infty))$,
    $$
    \int_{\R^N} u(t)\phi (t)=\int_{\R^N} f\phi(0)+ \int_0^t\int_{\R^N}\phi_t u+
    \frac{m_2}{2}\int_0^t\int_{\R^N}(\Delta \phi)(u-1)_+.
    $$
\end{proof}

We next study the limit $\eps\to0$ for the mushy region
$\mathcal{M}^\eps(t)$ associated to $u^\eps$ for initial data such
that $\mathcal{M}^\eps(0)=\emptyset$. We first estimate the size of the mushy region for fixed $\eps>0$.

\begin{theorem}\label{thm:mushy.disappear}
    Let $f\in\L^1_+(\R^N)$ be a nontrivial initial data such
    that $\Dsupp(f)=\Dsupp((f-1)_+)$. For any  $t>0$,
    \begin{equation}\label{eq:mushy.disappear}
      \mathcal{M}^\eps(t)\subset \{x\in\mathbb{R}^N:0<\dist(x,\Dsupp(v^\eps(\cdot,t))< \eps R_J \}.
    \end{equation}
\end{theorem}

\begin{proof}
    Let $x\notin \Dsupp(v^\eps(\cdot, t))+B_{\varepsilon R_J}$ for some fixed $t>0$. Since the support of $v^\eps(t)$ is nondecreasing, then for any
    $0\leq s\leq t$, $x\notin \Dsupp(v^\eps(\cdot, s))+B_{\varepsilon R_J}$.     Then  Lemma~\ref{lem:ut=0}
   implies that $u^\eps(x,t)=f(x)=0$, which implies~\eqref{eq:mushy.disappear}.
 \end{proof}

Unfortunately, since the convergence of the functions $\{u^\eps\}$
is rather weak, this is not enough to prove the convergence of the
mushy regions. However, we will be able to prove that their limsup,
$$
\mathcal{M}^*(t)=\limsup_{\eps\to0} \mathcal{M}^\eps(t) = \bigcap_{\eta>0}\bigcup_{\eta<\eps}\mathcal{M}^\eps(t),
$$
consisting of all points $x$ such that for any $\eta>0$ there exists
an $\eps\in(0,\eta)$ such that $x\in\mathcal{M}^\eps(t)$, is a
negligible set.

Recall that for the local problem, under our assumptions on the initial data no mushy regions are created, so that the supports of $v(t)$ and  $u(t)$ coincide for all $t\ge0$. Moreover, $v$ is continuous in $\R^N\times(0,\infty)$ \cite{CF}, hence the distributional support of $v(t)$ can be understood as the closure of the set $\{v(t)>0\}$.

\begin{corollary}
    Let $f\in\L^1_+(\R^N)$ such
    that $\Dsupp(f)=\Dsupp((f-1)_+)$. Then, for any $t>0$,
    $\mathcal{M}^*(t)$ has zero $N$-dimensional Lebesgue measure.
\end{corollary}

\begin{proof}
       We know that $u^\eps(\cdot,t)$ converges pointwise to $u(\cdot,t)$ except on a set $F$ which has zero $N$-dimensional Lebesgue measure.
    Apart from the set $F$ there are three possibilities:

    \noindent(i) $x\in\partial\{v(t)>0\}$. This set is known to have zero $N$-dimensional Lebesgue measure \cite{CR}.

    \noindent(ii) $x\in\{v(t)=0\}\setminus\partial\{v(t)>0\}$.
    Then since there are no mushy regions for the limit (local) equation here,
    necessarily, $u(x,t)=0$. This implies that for $\eps$ small enough (say less than $\eps_0$),
    we have $u^\eps(x,t)<1$, thus $v^\eps(x,t)=0$. Hence $x$ cannot belong to $\mathcal{M}^*(t)$
    because $x$ belongs to no mushy region for $\eps<\eps_0$, see Theorem~\ref{thm:mushy.disappear}.

    \noindent(iii) $x\in\{v(t)>0\}$. Then for $\eps$ small enough,
    $u^\eps(x,t)>1$ because it converges to $u(x,t)>1$. Thus, for $\eps$ small enough, such an $x$
    does not belong to any mushy region, hence it is not in the limsup

    Thus we have proved that $\mathcal{M}^*(t)$ is included in $F\cup\partial\{v(t)>0\}$ which is
    a negligible set.
\end{proof}

\section{Asymptotic behavior}
\label{sect:asymptotic} \setcounter{equation}{0}

Our next aim is to describe the large time behavior of the
solutions to our model.  For the local Stefan problem it is given by
a \lq mesa'-type problem \cite{GilQuirosVazquez2010}. To be more
precise, $u$ converges to $\tilde f=f+\Delta w$, where $w$ solves
the elliptic obstacle-type problem
$$
w\ge 0,\quad 0\le f+\Delta w \le 1,  \quad (f+\Delta w-1)w=0.
$$
In our case the limit is also given by a \lq mesa', but now of a
\emph{non-local} character, see below. This is to be contrasted with
the large time behavior of the non-local heat equation in the whole
space, which is given by the solution of the \emph{local} heat
equation with the same data, and hence by a multiple of the
fundamental solution of the latter equation.

\subsection{Formulation of the Stefan problem as a parabolic non-local obstacle problem (in complementarity form)}

We consider here a nonnegative initial $u$ giving rise to a nonnegative solution $u$.
To identify the asymptotic limit for $u$, we define the {\em
Baiocchi variable}
$$
w  (t)=\int_0^t v (s)\ds.
$$
A variable of this kind was first used by Baiocchi in 1971 to deal
with the dam problem \cite{Ba1}, \cite{Ba2}. The enthalpy and the
temperature can be recovered from $w$ through the formulas
\begin{equation}
\label{u.v.Stefan} u=f+J*w-w,\qquad v=\partial_tw,
\end{equation}
where the time derivative has to be understood in the sense of
distributions. Moreover,
$$
0\le u-v\le1,\quad (u-1-v)v=0\quad\text{a.e.}
$$

The distributional supports of $v$ and $w$ coincide for all times.
\begin{lemma}For any $t>0$, we have
$\Dsupp(v(t))=\Dsupp(w(t))$.
\end{lemma}
\begin{proof}
Let $x\in\Dsupp(w(t))$. Then, given $\varphi\in
\C_\textrm{c}^\infty(\mathbb{R}^N)$, $\varphi\ge 0$, $\varphi(x)>0$, we have
$$
0<\int_{\mathbb{R}^N}
w(t)\varphi=\int_0^t\left(\int_{\mathbb{R}^N}
v(s)\varphi\right)\ds.
$$
Hence, there exists $s<t$ such that $\int_{\mathbb{R}^N} v(s)\varphi>0$, i.e.,
 $x\in\Dsupp(v(s))$. Using the retention property for $v$ we finally get that $x\in\Dsupp(v(t))$.

Conversely, assume that $x\in\Dsupp(v(t))$. Then, given
$\varphi\in \C_\textrm{c}^\infty(\mathbb{R}^N)$, $\varphi\ge 0$, $\varphi(x)>0$, we have
$\int_{\mathbb{R}^N}  v(t)\varphi>0$. Since $v\in
\C([0,\infty);\L^1(\mathbb{R}^N))$, we  have that there exists a
value $\delta>0$ such that $\int_{\mathbb{R}^N}
v(s)\varphi>0$ for all $s\in(t-\delta,t)$. This implies that
$\int_{\mathbb{R}^N}  w(t)\varphi>0$, hence
$x\in\Dsupp(w(t))$.
\end{proof}
Thanks to this lemma, $(u-1-v)w=0$ a.e. Hence $w$ solves a.e.~the
complementarity problem
\begin{equation}
\label{variational.form.of.Stefan}  w \ge 0, \quad  0\le f+J*w-w
-\partial_tw \le 1,\quad (f+J*w-w -1-\partial_tw)w  =0,
\end{equation}
plus the initial condition $w(0)=0$. An analogous formulation for
the local Stefan problem was given in \cite{D}, see also \cite{FK}.

\subsection{A non-local elliptic obstacle problem}

If $\int_0^\infty\|v(t)\|_{\L^1(\mathbb{R}^N)}\dt<\infty$, then
$w(t)$  converges monotonically and in $\L^1(\mathbb{R}^N)$  as
$t\to\infty$ to
$$
w_\infty=\int_0^{\infty} v(s)\ds\in \L^1(\mathbb{R}^N).
$$
Thus, see~\eqref{u.v.Stefan}, $u(\cdot,t)$ converges point-wisely
and in $\L^1(\mathbb{R}^N)$ to
$$
\tilde f=f+J*w_\infty-w_\infty.
$$
Passing to the limit as $t\to\infty$
in~\eqref{variational.form.of.Stefan},
we get that $w_\infty$ is a solution with data $f$  to the \emph{nonlocal obstacle problem}:
\begin{equation} \label{forma.variacional.de.HS} \tag{OP}
\left\{\begin{array}{l}
\mbox{Given a non-negative data $f\in\L^1(\mathbb{R}^N)$, find a
non-negative
}\\
\mbox{function $w\in\L^1(\mathbb{R}^N)$ such that }\\[8pt]
 0\le f+J*w-w\le 1, \quad
\quad (f+J*w-w-1)w=0\quad \text{ a.e.}
\end{array}\right.
\end{equation}
This non-local obstacle problem has a unique solution. The proof is
based on the following Liouville type lemma for $J$-subharmonic
functions.
\begin{lemma}
\label{lemma:Liouville}Let $w\in\L^1(\mathbb{R}^N)$ such that
$w\ge0$, $w\le J*w$ a.e. Then $w=0$ a.e.
\end{lemma}
\begin{proof}
Assume first that $w$ is continuous, and fix $\eps>0$. Since $w$ is
integrable, there is a radius $R$ such that
$$
    \int_{|x|\ge R}w\le \frac{\eps}{\|J\|_{\L^\infty(\mathbb{R}^N)}}.
$$
Hence, for $|x|\ge R+R_J$
\begin{equation}
\label{eq:estimate.w.at.infinity}
    w(x)\le (J*w)(x)\le \|J\|_{\L^\infty(\mathbb{R}^N)}\int_{B_{R_J}(x)}\!\!\!w\le \|J\|_{\L^\infty(\mathbb{R}^N)}\int_{|x|\ge R}w \leq\eps.
\end{equation}
So, let us assume that for some $x\in\R^N$, $w(x)>\varepsilon$. Then
the maximum of $w$ is attained at some point $\bar x\in B_{R+R_J}$
and
$$
\max_{\R^N} w=w(\bar x)>\eps.
$$
Using that $w\le J*w$, we first deduce that $w(x)=w(\bar x)$ in
$B_{R_J}(\bar x)$ and then, spreading this property to all the space
by adding each time the support of $J$, we conclude that $w=w(\bar
x)>\eps$ in all $\R^N$. But this is a contradiction
with~\eqref{eq:estimate.w.at.infinity}. So, we deduce that $0\le
w\le\varepsilon$ for any $\varepsilon>0$, hence $w\equiv0$.

If $w$ is not continuous, we consider $w_n=w*\rho_n$, where $\rho_n$
is an approximation of the identity. The continuous function $w_n$
satisfies all the hypotheses of the lemma, hence $w_n=0$. Letting
$n\to\infty$ we obtain $w=0$ a.e.
\end{proof}

We will also need the following non-local version of Kato's
inequality (see \cite{Ka} for the local inequality),
\begin{equation}\label{eq:kato}
(J*w-w)\mathds{1}_{\{w>0\}}\le J*w_+-w_+\quad\text{a.e.},
\end{equation}
which is trivially valid for any function in $\L^1_{\rm loc}(\mathbb{R}^N)$.
\begin{theorem}
\label{thm:uniqueness.non-local.mesa} Problem
\eqref{forma.variacional.de.HS} has at most one solution.
\end{theorem}
\begin{proof} The key point is that solutions to
Problem~\eqref{forma.variacional.de.HS} satisfy
$$
  \tilde f=f+J*w-w, \quad \tilde f\in \beta(w) \quad \text{a.e.},
$$
where $\beta$ is the sign graph, see~\cite{BenilanBoccardoHerrero}
for the local case.

Let $w_i$, $i=1,2$, be two solutions
to~\eqref{forma.variacional.de.HS} with initial data $f$, and
$\tilde f_i$ be the corresponding projections.  Since $\tilde f_i\in
\beta(w_i)$, we have
$$
0\le(\tilde f_1-\tilde
f_2)\mathds{1}_{\{w_1>w_2\}}
=\big(J*(w_1-w_2)-(w_1-w_2)\big)\mathds{1}_{\{w_1>w_2\}}\quad\text{a.e.},
$$
from where we get, using Kato's inequality~\eqref{eq:kato}, that
$$
(w_1-w_2)_+\le J*(w_1-w_2)_+.
$$
Therefore, $(w_1-w_2)_+$ satisfies the hypotheses of
Lemma~\ref{lemma:Liouville}. We conclude that $w_1\le w_2$.
Interchanging the roles of $w_1$ and $w_2$, we get the result.
\end{proof}

\subsection{The mesa problem}
At this point we have a precise characterization  of the large time
behavior of solutions to the non-local Stefan
problem~\eqref{eq:stefan.nolocal} whenever~$\int_0^\infty
\|v(t)\|_{\L^1(\mathbb{R}^N)}\dt$ is finite. This is the
case, for instance, under the hypotheses of Lemma~{\rm
\ref{lem:cont.supp.v}}, see Corollary~\ref{corollary:exponential.decay}, or for general
$f\in\L^1_+(\R^N)$ if $N\ge3$, see Corollary~\ref{cor:decay.rate.integrable}.
\begin{theorem}
\label{thm:convergence.t.infty} Let $f\in\L^1_+(\R^N)$, and
assume in addition, if $N=1,2$,  the hypotheses of Lemma~{\rm
\ref{lem:cont.supp.v}} . If $u$ is the solution to
problem~\eqref{eq:stefan.nolocal} and $w$ is the solution of problem~\eqref{forma.variacional.de.HS}, then $u(t)\to f+J*w-w$ in
$\L^1(\mathbb{R}^N)$ as $t\to\infty$.
\end{theorem}

The map $\mathcal{P}: f\mapsto \tilde f=f+J*w-w$ projects the data
$f$ onto a \lq mesa'-type profile: $0\le \tilde f\le 1$, $\tilde
f=1$ in the \emph{non-coincidence set} $\{w>0\}$. Notice that
$\|\tilde f\|_{\L^1(\mathbb{R}^N)}=\| f\|_{\L^1(\mathbb{R}^N)}$.
However, in contrast with the local problem, the projection $\tilde
f$ is not necessarily equal to $f$ in the coincidence set $\{w=0\}$.

Up to now we have been able to prove the existence of a solution of
\eqref{forma.variacional.de.HS} for any $f\in\L^1_+(\mathbb{R}^N)$ only if $N\ge3$.
For low dimensions, $N=1,2$, we have needed to add the
hypotheses of Lemma~{\rm \ref{lem:cont.supp.v}}.  Hence, for low dimensions
the projection operator $\mathcal{P}$ is in principle only defined
under these extra assumptions. However,  $\mathcal{P}$ is continuous, in the
$\L^1$-norm, in the subset of~$\L^1_+(\mathbb{R}^N)$ of functions satisfying the hypotheses of
Theorem~\ref{thm:convergence.t.infty}.
\begin{corollary}
\label{corollary:contraction.mesa} Let
$f_i$, $i=1,2$, satisfying the hypotheses of Theorem~{\rm \ref{thm:convergence.t.infty}}. Then
$$
\|\tilde f_1-\tilde f_2\|_{\L^1(\mathbb{R}^N)}\le
\|f_1-f_2\|_{\L^1(\mathbb{R}^N)}.
$$
\end{corollary}
\begin{proof}
Since \eqref{forma.variacional.de.HS} has uniqueness, any solution
with initial data  satisfying the hypotheses of Theorem~\ref{thm:convergence.t.infty} can
be obtained as the limit as $t\to\infty$ of the solution with the same
initial data of the non-local Stefan problem. Hence the result
follows just passing to the limit as $t\to\infty$  in the
contraction property for this latter problem.
\end{proof}
Since the class of functions satisfying the hypotheses of Theorem~\ref{thm:convergence.t.infty} is dense in $\L^1_+(\mathbb{R}^N)$, we can extend $\mathcal{P}$ by continuity to the whole of
this bigger space. Thus, for any $f\in\L^1_+(\mathbb{R}^N)$, $\mathcal{P}f$
is the limit in $\L^1$ of $\{\mathcal{P}f_n\}$, where $\{f_n\}$ is
any sequence of nonnegative, measurable, bounded, and compactly supported
functions approximating $f$ in $\L^1(\mathbb{R}^N)$.

Let us notice that, though for any sequence of functions $\{f_n\}$
converging to $f$ in $\L^1(\mathbb{R}^N)$ we have convergence of
$\{\mathcal{P}f_n\}$, we are not able to prove the convergence of
the corresponding solutions $\{w_n\}$ to a solution
of~\eqref{forma.variacional.de.HS}, except under the hypotheses of Theorem~\ref{thm:convergence.t.infty}.
The main obstacle to prove this convergence is the lack of compactness of
the inverse of $\mathcal{L}$ (recall $\mathcal{L}=J\ast v-v$).

\subsection{Asymptotic limit for general data}
A simple argument now leads to the following characterization of the
asymptotic limit of the non-local Stefan problem for general
integrable initial data.

\begin{theorem}
Let $f\in \L^1_+(\mathbb{R}^N)$, and $u$ the corresponding solution to
problem~\eqref{eq:stefan.nolocal}.  Let $\mathcal{P}f$ be the
projection of $f$ onto a non-local mesa. Then $u(\cdot,t)\to
\mathcal{P}f$ in $\L^1(\mathbb{R}^N)$ as $t\to\infty$.
\end{theorem}
\begin{proof}
Given $f$, let $\{f_n\}\subset \L^1(\mathbb{R}^N)$ be a sequence of functions satisfying the
hypotheses of Theorem~\ref{thm:convergence.t.infty} which approximate $f$ in
$\L^1(\mathbb{R}^N)$. Let $u_n$ be the corresponding solutions to
the non-local Stefan problem. We have,
$$
\|u(t)-\mathcal{P}f\|_{\L^1(\mathbb{R}^N)}\le\|u(t)-u_n(t)\|_{\L^1(\mathbb{R}^N)}+\|u_n(t)-\mathcal{P}
f_n\|_{\L^1(\mathbb{R}^N)}+\|\mathcal{P} f_n-\mathcal{P}f\|_{\L^1(\mathbb{R}^N)}.
$$
Using the contraction property for the non-local Stefan problem, and
the large time behavior for bounded and compactly supported
initial data,
$$
\limsup_{t\to\infty}\|u(t)-\mathcal{P}f\|_{\L^1(\mathbb{R}^N)}\le\|f-f_n\|_{\L^1(\mathbb{R}^N)}+\|\mathcal{P}
f_n-\mathcal{P}f\|_{\L^1(\mathbb{R}^N)}.
$$
Letting $n\to\infty$ we get the result.
\end{proof}

\section{Numerical experiments}
In order to illustrate some of the previous results concerning solutions of~\eqref{eq:stefan.nolocal}, we show some numerical experiments.  We will take $f\in{\L^1_+(\mathbb{R}^N)}$ compactly supported, and $J(x)=0.75(1-x^2)_+$.
The space discretization  is implemented using the trapezoidal rule.
For the integration in time we have used an ODE integrator provided
by {\sc Matlab}$^\circledR$.

 \subsection{Creation of mushy regions}
In Figure~\ref{fig.mushy} we illustrate the creation of mushy
regions. As mentioned in the introduction, this phenomenon  is
absent in the local problem: if there are no mushy regions
initially, this is also true for any later time. On the contrary, in
our non-local model, regions with $u$ between $0$ and $1$ appear in
the neighborhood of the water region as time passes. This is one of
the main qualitative features of the model.  For this simulation we
have taken $f(x)=2\cdot\mathds{1}_{[-1,1]}$.

\begin{center}
  \begin{figure}
    [ht]
    \includegraphics[width=10cm]{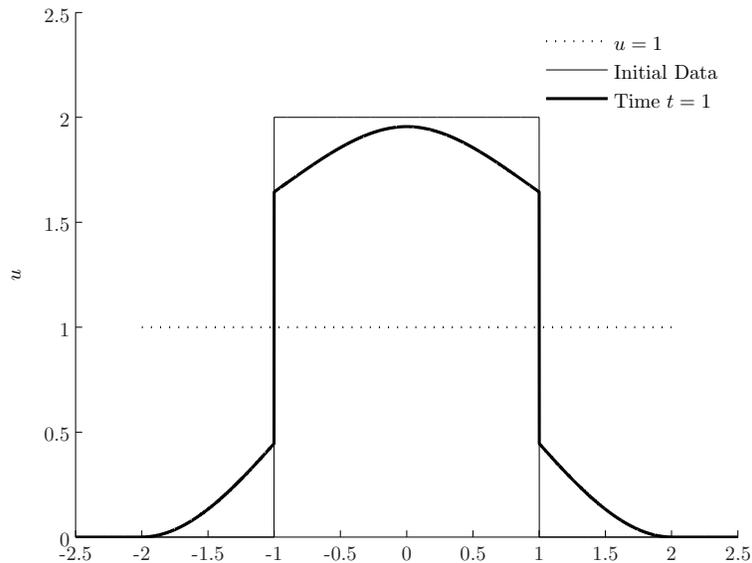}
    \caption{Creation of mushy regions.}
    \label{fig.mushy}
  \end{figure}
\end{center}

\subsection{Appearance of disconnected water regions} In the local
case, it never happens that a new water region appears disconnected
from the ones that were already present immediately before. On the
contrary, this indeed happens sometimes in our non-local model, as
shown in Section~\ref{subsect:disconnected}. We have exemplified
this fact in figures~\ref{fig.mushy2} and~\ref{fig.mushy2.zoom},
which correspond to an initial datum which is the sum of two
characteristic functions,
$f(x)=2.5\cdot\mathds{1}_{[-1.25,-0.5]}+0.99\cdot\mathds{1}_{[0.5,1]}$.
\begin{center}
  \begin{figure}
    [ht]
    \includegraphics[width=10cm]{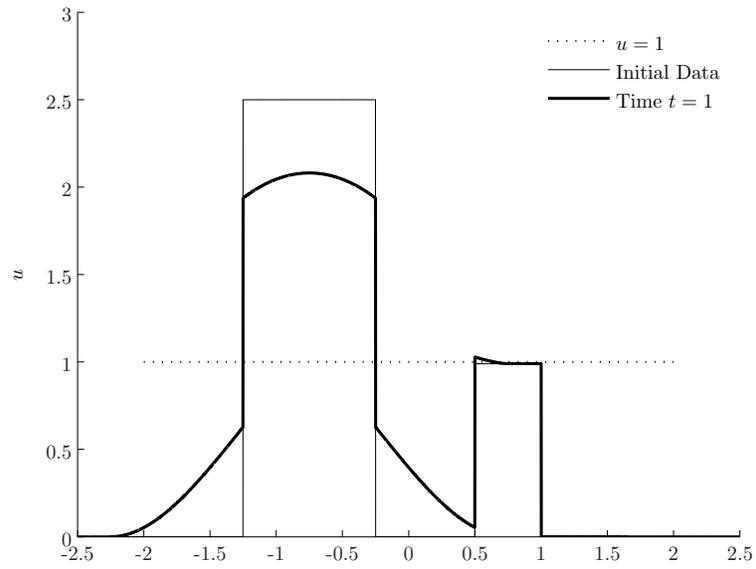}
    \caption{Appearance of disconnected water regions.}
    \label{fig.mushy2}
  \end{figure}
\end{center}
\begin{center}
  \begin{figure}
    [ht]
    \includegraphics[width=10cm]{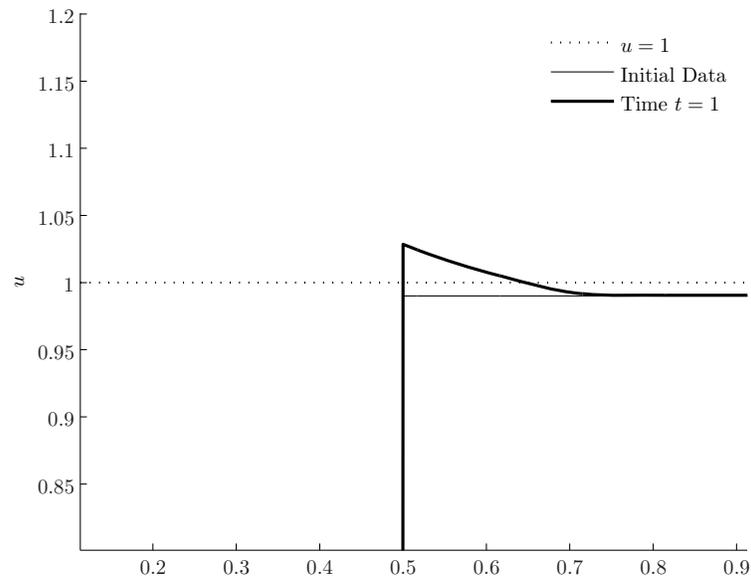}
    \caption{Appearance of disconnected water regions. Zoom.}
    \label{fig.mushy2.zoom}
  \end{figure}
\end{center}

\subsection{Behavior as $\varepsilon\to 0$}
In Figure~\ref{fig.eps.to.0} we illustrate the effect described in
Section~\ref{sect:limit.epsilon}. When $\varepsilon\to 0$, the
solution of the non-local problem converges to the solution of the
local Stefan problem with the same initial data. Moreover, the mushy
regions that were created because of the non-local effect disappear.
The initial datum is again the characteristic function
$f(x)=2\cdot\mathds{1}_{[-1,1]}$. We compute the solution for
$\eps=1$, $0.5$ and $0.2$.
\begin{center}
  \begin{figure}
    [ht]
    \includegraphics[width=10cm]{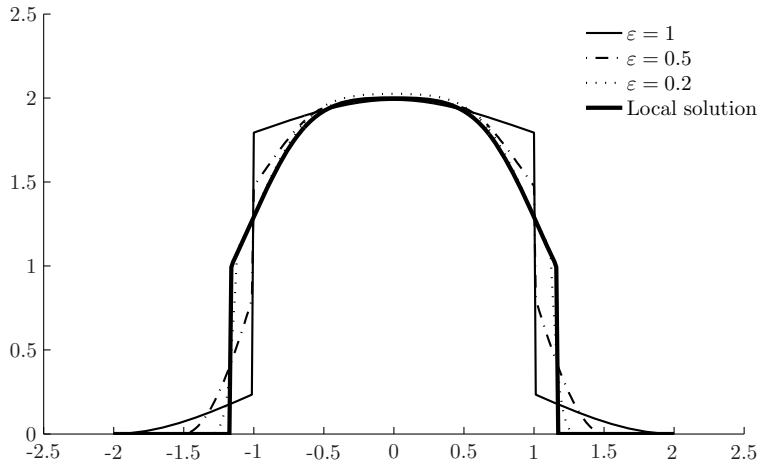}
    \caption{Convergence to the solution of the local problem as $\varepsilon\to 0$.}
    \label{fig.eps.to.0}
  \end{figure}
\end{center}


%
\subsection{Behavior as $t\to \infty$}
Finally we show an example of the asymptotic behavior of the
solutions as $t\to\infty$, which is described in
Section~\ref{sect:asymptotic}. We have taken here as initial datum
$$
  f(x)=\left\{
  \begin{array}
    {l@{\qquad}l}
    0,& x< -4,\\
    (\sin(5x))_+,&-4\leq x\leq -1.5,\\
    \sin(2x)+3,&|x|< 1.5,\\
    0, &1.5\leq x\leq6,\\
    0.3,&6<x<6.5,
  \end{array}
  \right.
$$
so that the convergence to a non-local mesa is more evident.

\begin{center}
  \begin{figure}
    [ht]
    \includegraphics[width=10cm]{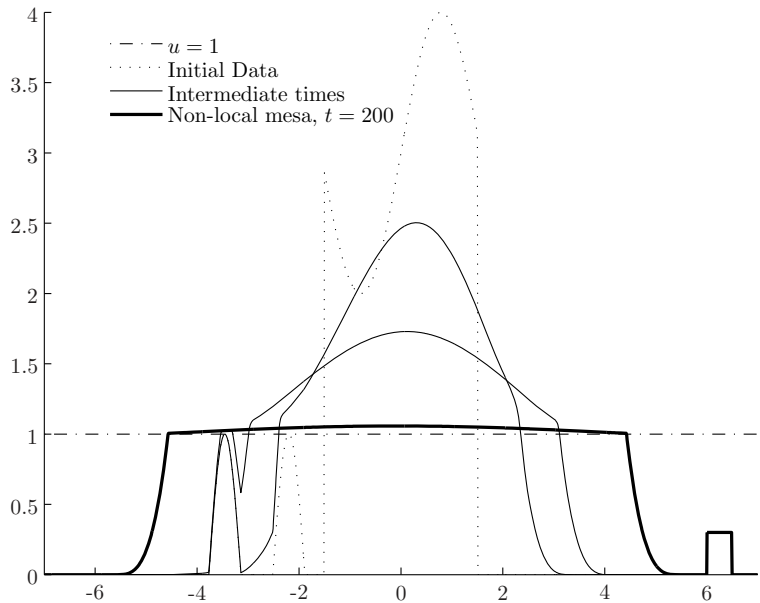}
    \caption{Convergence to the mesa as $t\to\infty$.}
    \label{fig.t.infty}
  \end{figure}
\end{center}

\section{Conclusion}

We have proposed a non-local model to describe the evolution of a
mixture of ice and water in an intermediate  mesoscopic scale, and
derived some properties for the solutions which are interesting from
the physical point of view, in particular the creation of mushy
regions. We have also proved that the local Stefan problem is
obtained in the macroscopic limit, and that mushy regions disappear
in this latter scale if they were not present initially.

Theorem \ref{thm:mushy.creation} does not only  assert that mushy
regions are indeed created. It has another interesting consequence
reflected in \eqref{eq:equality.supports.short.times}: there is a
waiting time $t_0=t_0(f)$ until which the support of $v=(u-1)_+$
remains identical to that of $(f-1)_+$. In other words, the support
of the water phase does not evolve until time $t_0$, though the
temperature in this phase is decreasing because energy is consumed
to break the ice. This waiting time can be interpreted as the time
needed to break the nearby ice phase in order to convert it into
water. We are facing  a typical ``latent heat" phenomenon which is
not included in the usual local model. In the case of
equation~\eqref{eq:stefan.local}, the water phase begins to move
instantaneously and ice is melted without any waiting time.

Summarizing, this nonlocal model seems meaningful and allows several
natural phenomena at an intermediate scale which are not present in
the local model.
\section*{Appendix}
\renewcommand{\d}{\,\mathrm{d}}
\renewcommand{\theequation}{A.\arabic{equation}}
\renewcommand{\thesection}{A}
{\small
In  the spirit of \cite{ChasseigneChavesRossi2006}, we prove asymptotic estimates concerning the decay of solutions of the nonlocal heat equation
\begin{equation}\label{app:eq:nonloc.heat}
 u_t=J\ast u-u,\qquad u(0)=f.
\end{equation}
We only assume here that the kernel $J$ is symmetric, nonnegative,
with unit mass and a finite second order momentum denoted by $m_2$.
We prove that the solution $u$ is asymptotically similar to the
solution of the local heat equation with the same initial data. The
gain with respect to  the paper~\cite{ChasseigneChavesRossi2006} is
that our results are valid for general  initial data $f\in\L^1(\mathbb{R}^N)$, without
assuming anything about their Fourier transform or their
$\L^\infty$-norm.

Notice that the solutions will not eventually enter in
the class of data considered in~\cite{ChasseigneChavesRossi2006}
(unless they were already there initially). Indeed, they can be
written as
\begin{equation}
\label{eq:representation.formula.nlhe}
 u(t)=\e^{-t}f+\big(\omega(t)\ast f\big)
\end{equation}
with $\omega(t)$ smooth, integrable and
bounded,~\cite{BrandleChasseigneFerreira2010}. Hence, $u$ is not
bounded if $f\notin\L^\infty(\mathbb{R}^N)$. However, by subtracting
$\e^{-t}f$ to $u(t)$ we ensure that $u(t)-\e^{-t}f$ is bounded,
since $\omega(t)$ is bounded and $f\in\L^1(\mathbb{R}^N)$. Moreover,
$$\|\hat u(t)-\e^{-t}\hat f\|_{\L^1(\mathbb{R}^N)}\leq\|\hat{\omega}(t)\|_{\L^1(\mathbb{R}^N)}\|\hat{f}\|_{\L^\infty(\mathbb{R}^N)}\leq
    \|\hat{\omega}(t)\|_{\L^1(\mathbb{R}^N)}\|f\|_{\L^1(\mathbb{R}^N)}<\infty,$$
so that the Fourier transform of the difference is in
$\L^1(\mathbb{R}^N)$. This is important to go back to the original
variables after making the computations in Fourier variables.

\begin{theorem}\label{app:thm:refined}
    Let $f\in\L^1(\mathbb{R}^N)$ and $u$ be the solution of \eqref{app:eq:nonloc.heat} with $u(0)=f$. Let  $h$
   be the solution of
    $$
   h_t=\displaystyle\frac{m_2}{2}\Delta h,\qquad h(0)=u(0)=f.
    $$
    Then, as $t\to\infty$, there exists a function $\eps(t)\to0$  (depending only on $J$
    and $N$) such that
    \begin{eqnarray}
    \label{eq:app.estimate.infty}
     t^{N/2}\max_{\R^N}\big|u(t)-\e^{-t}f-h(t)\big|\leq
     \|f\|_{\L^1(\mathbb{R}^N)}\eps(t).
    \end{eqnarray}
\end{theorem}

\begin{proof}
    Following \cite[Theorem~2.2]{ChasseigneChavesRossi2006}, we start from
    \begin{equation*}
    \hat{u} (\xi, t) = \e^{ (\hat{J} (\xi) -1)t} \hat{f}(\xi)\quad\text{and}\quad
    \hat{h} (\xi , t) = \e^{ - c |\xi|^2 \, t}\hat{f}(\xi),\ \mbox{ with }c=\frac{m_2}{2},
    \end{equation*}
    but we make a different estimate using $u(t)-\e^{-t}f$ instead
 of $u(t)$,
    \begin{eqnarray*}
    \displaystyle \int_{\R^N} |\hat{u}-\e^{-t}\hat{f}-\hat{f}| (\xi, t) \d\xi &=&
    \int_{\R^N} \left|\left( \e^{t (\hat{J} (\xi) -1)} - \e^{-t}
    -\e^{- c|\xi|^2 t}\right) \hat{f}(\xi) \right| \d\xi \\[10pt]
    &=& \displaystyle \int_{|\xi| \ge r(t)} \left|\left( \e^{t
    (\hat{J} (\xi) -1)} - \e^{-t}-\e^{- c |\xi|^2 t}\right)
    \hat{f}(\xi) \right| \d\xi \\[10pt]
    &&+\displaystyle \int_{|\xi | < r(t)} \left|\left( \e^{t
    (\hat{J} (\xi) -1)} -\e^{-t} - \e^{- c |\xi|^2 t}\right)
    \hat{f}(\xi) \right| \d\xi
    \\[10pt]&=& I + I\!I.
    \end{eqnarray*}
As in~\cite{ChasseigneChavesRossi2006}, we split $I$ into two parts
\begin{equation*}
    I \le  \int_{|\xi| \ge r(t)} \left| \e^{- c|\xi|^2t} \hat{u}_0(\xi)\right| \d\xi
    + \int_{|\xi| \ge r(t)} \left| \e^{t (\hat{J} (\xi) -1)}-\e^{-t}\right| | \hat{u}_0(\xi) | \d\xi
    = I_1 + I_2.
    \end{equation*}
     The only important modification with respect to the proof of~\cite[Theorem~2.2]{ChasseigneChavesRossi2006} concerns the term $I_2$.
     In~\cite{ChasseigneChavesRossi2006}
 this term is estimated by using the $\L^1$-norm of $\hat{f}$, which is something we do not want to do here.

 In order to estimate $I_2$ we use that that $\hat{J}$ verifies
    \begin{equation*}
    \hat{J}(\xi)\le 1-c|\xi|^2+|\xi|^2 h(\xi), \quad \text{with $h$ bounded,\quad $h(\xi)\to 0$ as $\xi\to 0$.}
    \end{equation*}
    Hence
    there exist $a,D,\delta>0$ such that
    \begin{equation*}
    \hat{J} (\xi) \le 1-D|\xi|^2, \mbox{ for } |\xi|\le a\quad \mbox{ and }\quad
    |\hat{J} (\xi)| \le 1 - \delta, \mbox{ for } |\xi|\ge a.
    \end{equation*}
    We decompose $I_2$ by considering separately the sets $\{r(t)\leq|\xi|\le a\}$ and
    $\{|\xi|\ge a\}$.
   The integration over $\{|\xi|\geq a\}$ is estimated here taking into account the term $\e^{-t}f$,
   $$\int_{|\xi| \ge a} \left| \e^{t (\hat{J} (\xi) -1)}
    -\e^{-t}\right| |\hat{f}(\xi) | \d\xi =\e^{-t}\int_{|\xi| \ge a} \left| \e^{t \hat{J} (\xi)}
    -1\right| |\hat{f}(\xi) | \d\xi.
    $$
    Using that in this set $|\hat{J}|\leq 1-\delta$, we get that
  \begin{eqnarray*}
    \left| \e^{t \hat{J} (\xi)}
    -1\right| |\hat{f}(\xi) |&\leq& |\hat{f}(\xi)|\sum_{n=1}^\infty\frac{t^n|\hat{J}(\xi)|^n}{n!}\\
    &\leq& t|\hat{J}(\xi)||\hat{f}(\xi)|\sum_{n=0}^\infty\frac{t^n(1-\delta)^n}{(n+1)!}\\
    &\leq& t\e^{(1-\delta) t}|\hat{J}(\xi)\vvvert\hat{f}\|_\infty.
\end{eqnarray*}
    Hence
    $$t^{N/2}\int_{|\xi| \ge a} \left| \e^{t (\hat{J} (\xi) -1)}
    -\e^{-t}\right| |\hat{f}(\xi) | \d\xi \leq t^{N/2+1}\e^{-\delta t}\|\hat{f}\|_\infty
    \int_{|\xi|\geq a}|\hat{J}(\xi)|\d\xi\leq \|f\|_1\eps_2(t),
    $$
    where $\eps_2(t)\to0$ exponentially fast, and independently of the initial data.

Summing up (see~\cite[Theorem~2.2]{ChasseigneChavesRossi2006} for the estimates of the other terms of $I + I\!I$)
 $$t^{N/2}\|\hat{u}(t)-\e^{-t}\hat{f}-\hat{h}(t)\|_{\L^1(\mathbb{R}^N)}\leq\|f\|_{\L^1(\mathbb{R}^N)}\eps(t)\to0$$
 as $t\to\infty$ for some function $\eps$ which only depends on $J$ and $N$. This implies~\eqref{eq:app.estimate.infty} by going back to the original variables.
\end{proof}


\end{document}